\documentclass[a4paper,11pt]{article}

\usepackage{amsmath,amsthm,amssymb,dsfont,comment,scalerel,hyperref,xcolor,mathtools,xparse,relsize,stmaryrd,mathrsfs,euscript,enumitem,authblk}
\usepackage{geometry,tikz}
\usepackage{tikz-3dplot}
\usetikzlibrary{calc}
\usepackage{ccaption}

\usepackage{Baskervaldx} 

\usepackage{hyperref}
\usepackage{comment}
\usepackage{xcolor}
\definecolor{unbleu}{rgb}{0.03, 0.15, 0.4}
\definecolor{monvert}{rgb}{0.0,.5,0.0}
\definecolor{britishracinggreen}{rgb}{0.0, 0.26, 0.15}
\definecolor{monbleu}{rgb}{0,.2,.8}
\definecolor{monautrebleu}{rgb}{0,0.4,.75}
\definecolor{applegreen}{rgb}{0.55, 0.71, 0.0}
\definecolor{monrouge}{rgb}{0.8, 0.0, 0.0} 
\definecolor{cadmiumgreen}{rgb}{0.0, 0.42, 0.24}
\definecolor{royalblue(traditional)}{rgb}{0.0, 0.14, 0.4}
\definecolor{black}{rgb}{0.0, 0.0, 0.0}
\definecolor{sepia}{rgb}{0.44, 0.26, 0.08}
\definecolor{teagreen}{rgb}{0.82, 0.94, 0.75}
\definecolor{yellow-green}{rgb}{0.6, 0.8, 0.2}
\definecolor{azure(colorwheel)}{rgb}{0.0, 0.5, 1.0}
\definecolor{awesome}{rgb}{1.0, 0.13, 0.32}
\definecolor{cadmiumyellow}{rgb}{1.0, 0.96, 0.0}
\definecolor{carrotorange}{rgb}{0.93, 0.57, 0.13}
\definecolor{green-yellow}{rgb}{0.68, 1.0, 0.18}
\definecolor{huntergreen}{rgb}{0.21, 0.37, 0.23}

\hypersetup{
pdfborder = {0 0 0},
colorlinks,
linkcolor=unbleu,
citecolor=unbleu,
urlcolor=unbleu
}

\newtheorem{definition}{Definition}[section]
\newtheorem{lemma}{Lemma}[section]
\newtheorem{cor}{Corollary}[section]
\newtheorem{proposition}{Proposition}[section]
\newtheorem{remark}{Remark}[section]
\newtheorem{theo}{Theorem}[section]
\newtheorem{question}{Question}[]

\newtheorem*{notation}{Notation}

\newcommand{\1}{\mathds{1}}
\DeclareMathOperator{\dep}{\mathrm{dep}}
\newcommand{\E}{\mathds{E}}
\newcommand{\Z}{\mathds{Z}}
\newcommand{\R}{\mathds{R}}
\newcommand{\N}{\mathds{N}}
\newcommand{\V}{\mathds{V}}
\newcommand{\Proba}{\mathds{P}}
\newcommand{\Zd}{\ensuremath{\Z^d}}    
\newcommand{\db}{\bar{\mathrm{d}}}

\newcommand{\Law}{{\mathrm{Law}}}
\newcommand{\Cov}{{\mathrm{Cov}}}
\newcommand{\Var}{{\mathrm{Var}}}
\newcommand{\sC}{\scaleto{C}{6pt}} 

\newcommand{\e}{\operatorname{e}}

\newcommand{\comp}{\mathrm{c}}

\DeclareMathAlphabet{\dutchcal}{U}{dutchcal}{m}{n}
\SetMathAlphabet{\dutchcal}{bold}{U}{dutchcal}{b}{n}
\DeclareMathAlphabet{\dutchbcal} {U}{dutchcal}{b}{n}
\DeclareMathOperator{\ent}{\scaleobj{1.15}{\dutchcal{h}}} 

\DeclareMathOperator{\dd}{\textup{d}\!}
\DeclareMathOperator{\eqlaw}{\stackrel{\scriptscriptstyle{law}}{=}}

\NewDocumentCommand{\bracksz}{ O{13pt} m }{%
  {\mathord{\scaleto{[}{#1}}#2\mathord{\scaleto{]}{#1}}}%
}

\NewDocumentCommand{\parensz}{ O{13pt} m }{%
  {\mathord{\scaleto{(}{#1}}#2\mathord{\scaleto{)}{#1}}}%
}

\NewDocumentCommand{\bracesz}{ O{13pt} m }{%
  {\mathord{\scaleto{\{}{#1}}#2\mathord{\scaleto{\}}{#1}}}%
}

\NewDocumentCommand{\anglesz}{ O{13pt} m }{%
  {\mathord{\scaleto{\langle}{#1}}#2\mathord{\scaleto{\rangle}{#1}}}%
}

\begin{document}

\title{Finitary coding and Gaussian concentration\\ for random fields}

\author[1]{J.-R. Chazottes}
\author[2]{S. Gallo}
\author[1,3]{D. Y. Takahashi}

\affil[1]{Centre de Physique Th\'eorique, CNRS, Institut Polytechnique de Paris, France}
\affil[2]{Departamento de Estat\'istica, Universidade Federal de S\~ao Carlos, S\~ao Paulo, Brazil}
\affil[3]{Instituto do C\'erebro, Universidade Federal do Rio Grande do Norte, Natal, Brazil}

\date{\today}

\maketitle

\begin{abstract}
We study Gaussian concentration inequalities for random fields obtained as finitary codings of i.i.d.\ fields, thereby linking concentration properties to the structure of finitary codings. A finitary coding represents a dependent random field as a shift-equivariant image of an i.i.d.\ process, where each output coordinate depends on a finite but configuration-dependent portion of the input. Gaussian concentration corresponds to uniform sub-Gaussian fluctuation bounds for all local observables.

Our main abstract result shows that Gaussian concentration is preserved under finitary codings of i.i.d.\ fields provided the coding volume has finite second moment. The proof relies on a refinement of the bounded-differences inequality, due to Talagrand and Marton, which accommodates configuration-dependent influences. Under an additional structural assumption, the short-range factorization property, satisfied in particular by codings arising from coupling-from-the-past constructions, a finite first moment suffices. We also show that these moment conditions are sharp.

Our abstract results yield a unified treatment of Gibbs measures and Markov random fields on $\mathbb Z^d$, and a large class of one-dimensional stochastic processes. Building on recent constructions of finitary codings for such models, notably by Spinka and collaborators, we obtain sharp necessary and sufficient conditions for Gaussian concentration for classical lattice models, including the Ising, Potts, and random-cluster models, showing that it holds if and only if the model lies in the full uniqueness regime. This significantly strengthens previous results, which were confined to strict subregimes of uniqueness, and in particular allows us to treat models that were beyond the reach of earlier methods.
In one dimension, we cover a large class of processes, including chains with unbounded memory. In the special case of countable-state Markov chains, we obtain equivalent characterizations in terms of geometric ergodicity, exponential return-time tails, and the existence of finitary i.i.d.\ codings with exponential tails.

\medskip

\end{abstract}

{\small 
\tableofcontents
}


\section{Introduction}

Gaussian concentration inequalities provide uniform control on the fluctuations of local observables of a random field. They assert that every local function with bounded single-site oscillations exhibits sub-Gaussian deviations, with constants independent of the size of its dependence set. Such bounds play a central role in probability theory, with applications in statistics, information theory, statistical mechanics, and ergodic theory. We refer to the literature cited below for further background.

A natural question is how Gaussian concentration behaves under the introduction of dependencies. In particular, under which conditions is Gaussian concentration preserved when an i.i.d.\ random field is transformed by a local, shift-equivariant map? The purpose of this paper is to address this question in the framework of \emph{finitary codings}.

Finitary codings originate in Ornstein's theory of Bernoulli shifts, where dependent processes are shown to be isomorphic to i.i.d.\ ones. A factor of an i.i.d.\ process is defined via a shift-equivariant measurable map, which may depend on the entire configuration. A finitary coding is a stronger notion, in which the value at each site is determined, almost surely, by inspecting only a finite (but random) region of the input configuration. This distinction is particularly relevant for lattice systems. While the plus state of the Ising model is always a factor of an i.i.d.\ process, it is a finitary factor if and only if the Gibbs measure is unique \cite{steif/vanderberg/1999}. Thus, finitary codings capture phase transitions, in contrast to the factor-of-i.i.d.\ notion. Further developments by Spinka and collaborators \cite{spinkaEJP,spinka-AoP,spinka-harel} provided general constructions of finitary codings together with quantitative control on coding radii.

Our contributions are as follows.

We first note that Gaussian concentration has nontrivial structural consequences: for shift-invariant finite-valued random fields, it implies Bernoullicity. While not one of our main results, this observation appears to be new.

We then establish general conditions under which Gaussian concentration is preserved under finitary codings. Our main results show that if a random field taking values in a standard Borel space is obtained as a finitary coding of an i.i.d.\ field, then Gaussian concentration holds whenever the associated coding volume has finite second moment. We further show that, under an additional structural assumption satisfied in particular by coupling-from-the-past constructions, this condition can be relaxed to finiteness of the first moment.

Finally, we apply these results to a broad class of models, including probabilistic cellular automata, Gibbs measures, and chains with unbounded memory. This yields concentration results beyond classical regimes such as Dobrushin uniqueness.

The proof of our main result relies on a concentration inequality 
originally due to Talagrand and subsequently sharpened by Marton via a conditional transportation inequality, which controls fluctuations in terms of expected squared influences. Classical bounded-differences inequalities are not suited to our setting, as the influence of a given input variable is random and configuration-dependent. Marton's inequality allows us to express these influences in terms of overlaps of random coding windows, leading naturally to a second-moment condition on the coding volume. Under an additional structural assumption, satisfied in particular by coupling-from-the-past constructions, this condition can be weakened to a first-moment bound. We also address the sharpness of these conditions.

Our abstract results apply in particular to finitary codings constructed via probabilistic cellular automata and coupling-from-the-past algorithms. Combined with existing constructions \cite{vanderberg/steif/1999,spinkaEJP,spinka-AoP,spinka-harel}, this yields Gaussian concentration for a wide range of Gibbs measures and related models.

At a conceptual level, our results support the following picture: in uniqueness regimes, both finitary codings of i.i.d.\ fields and Gaussian concentration hold, whereas in coexistence regimes, neither does. In all examples we consider in dimension $d \ge 2$, the coding radius has exponential or stretched-exponential tails, so that the required moment conditions are easily satisfied.

We also treat one-dimensional processes, including chains with unbounded memory. In particular, for irreducible and aperiodic countable-state Markov chains, our results yield several equivalent characterizations of Gaussian concentration, including geometric ergodicity, exponential return-time tails, and the existence of a finitary coding of an i.i.d.\ process.

Finally, we present several open problems. In particular, it remains open whether Gaussian concentration implies the existence of a finitary i.i.d.\ coding under suitable moment conditions in higher dimensions.

The paper is organized as follows.
Section~\ref{sec:setting} introduces configuration spaces, finitary codings, and Gaussian concentration bounds, and establishes several structural consequences of Gaussian concentration in the finite-valued setting (see Subsection~\ref{subsec:fvrf}). 
Section~\ref{sec:main} contains the main abstract results relating Gaussian concentration to moment conditions on the coding volume.
Section~\ref{sec:applications} is devoted to applications to concrete models.
Finally, Section~\ref{sec:remarks-and-questions} discusses optimality issues and presents a number of open problems.

\section{Configuration spaces, finitary codings, and Gaussian concentration
}\label{sec:setting}

\subsection{Configuration spaces and finitary codings}

As the concepts in this section lie at the intersection of ergodic theory, information theory, and stochastic processes, we will freely use the terminology and notation of all three fields.

Fix $d\ge1$. Let $(A,\mathcal F)$, $(B,\mathcal G)$  be standard Borel spaces (finite alphabets with the discrete topology are a special case) and consider
the configuration spaces
\[
A^{\Zd}=\{x=(x_i)_{i\in\Zd}:x_i\in A\},\qquad
B^{\Zd}=\{y=(y_j)_{j\in\Zd}:y_j\in B\},
\]
with the product $\sigma$-algebras. 
For $j\in\Zd$, we denote by
\[
T^j:A^{\Zd}\to A^{\Zd},\qquad
S^j:B^{\Zd}\to B^{\Zd}
\]
the shift operators acting by translation of coordinates,
\[
(T^j x)_i = x_{i+j}, \qquad (S^j y)_i = y_{i+j}, \qquad i\in\Zd.
\]
We use the $\ell^\infty$ norm $\|i\|_\infty=\max_{1\le k\le d}|i^{(k)}|$ and the closed $\ell^\infty$-balls
\[
B_\infty(j,r)=\{i\in\Zd:\|i-j\|_\infty\le r\}\,.
\]
We will denote its cardinality by $|B_\infty(j,r)|=(2r+1)^d$.
We use $\Lambda\subset\Z^d$ for a generic subset, and write $\Lambda\Subset\Z^d$ to indicate that $\Lambda$ is finite.

\begin{definition}[Coding map and coding radius]

A measurable $\varphi:A^{\Zd}\to B^{\Zd}$ such that $\varphi\circ T^j=S^j\circ \varphi$, for all $j\in\Zd$, is called a coding map. For $x\in A^{\Zd}$ we define the (pointwise) \emph{coding radius at the origin}
\[
r\!_\varphi(x)\;:=\;\inf\Big\{r\in\N_0:\ \forall x'\in A^{\Zd},\ 
x'_{B_\infty(0,r)}=x_{B_\infty(0,r)}\ \Rightarrow\ \varphi(x')_0=\varphi(x)_0\Big\}.
\]
If the set is empty then the coding radius is infinite.
\end{definition}
By shift-equivariance, the radius at site $j$ is $r\!_\varphi(T^j x)$ and $\varphi(x)_j$ depends only on
$x|_{B_\infty(j,r\!_\varphi(T^j x))}$. 

Let $\mu$ be a $T$-invariant probability measure on $A^{\Zd}$.
In ergodic-theoretic terminology, the triple
\[
\big(A^{\Zd},(T^j)_{j\in\Zd},\mu\big)
\]
is called a (measure-theoretic) shift dynamical system.
If $\varphi:A^{\Zd}\to B^{\Zd}$ is a coding map, then the pushforward measure
$\nu:=\varphi_*\mu$ is $S$-invariant on $B^{\Zd}$.
This defines another shift dynamical system $\big(B^{\Zd},(S^j)_{j\in\Zd},\nu\big)$,
which is called a \emph{factor} of $\big(A^{\Zd},(T^j)_{j\in\Zd},\mu\big)$.

\medskip

An equivalent formulation is in terms of canonical random fields.
Given $\mu$ as above, let $X=(X_i)_{i\in\Zd}$ be the canonical $A$-valued random field
on $(A^{\Zd},\mu)$, defined by $X_i(x)=x_i$.
We use the same notation for the natural action of the shift on random fields:
for $j\in\Zd$,
\[
(T^j X)_i := X_{i+j}.
\]
With this convention, $X$ is shift-invariant in law,
\[
T^j X \eqlaw X,\qquad j\in\Zd,
\]
and we will simply say that $X$ is shift-invariant.
If $\varphi:A^{\Zd}\to B^{\Zd}$ is a coding map, then
$Y:=\varphi(X)$ is the canonical $B$-valued random field under $\nu=\varphi_*\mu$,
and $Y$ is shift-invariant under $S$.
In this case, one says that $Y$ is a \emph{coding} of $X$.
The pointwise coding radius $r\!_\varphi(x)$ becomes the random variable
$r\!_\varphi(X)$ on $(A^{\Zd},\mu)$.

\medskip

We will be particularly interested in coding maps that are finitary.

\begin{definition}[Finitary coding / finitary factor]
With the notation introduced above, a coding map $\varphi$ is said to be
\emph{finitary} if $r\!_\varphi(x)<\infty$ for $\mu$-almost every $x$, or equivalently,
if $r\!_\varphi(X)<\infty$ almost surely.
In this case, $Y=\varphi(X)$ is called a \emph{finitary coding} of $X$,
and equivalently the shift dynamical system
$\big(B^{\Zd},(S^j)_{j\in\Zd},\nu\big)$ is called a \emph{finitary factor} of
$\big(A^{\Zd},(T^j)_{j\in\Zd},\mu\big)$.
\end{definition}

\begin{remark}
A \emph{block code} is the special case in which the coding radius $r\!_\varphi$
is bounded deterministically.
A classical example is provided by hidden Markov chains, obtained as functions of
finite-state Markov chains.
Finitary codings allow for unbounded coding radii, but require $r\!_\varphi<\infty$
almost surely.
\end{remark}

\medskip

Our primary focus is on the situation where $Y$ is obtained as a finitary coding of an i.i.d.\ random field. In other words, we study dynamical systems $\big(B^{\Zd}, (S^j)_{j\in\Zd}, \nu\big)$ that are finitary factors of a $d$-dimensional Bernoulli shift. 

\begin{definition}[i.i.d.\ random field and Bernoulli shift]
Let $(A,\mathcal F)$ be a standard Borel space with probability law $\varrho$.
An \emph{i.i.d.\ random field} is a family $X=(X_i)_{i\in\Zd}$ of $A$-valued random variables
that are independent and identically distributed with law $\varrho$.
Equivalently, the joint law of $X$ on $A^{\Zd}$ is the product measure
$\varrho^{\otimes \Zd}$.
The associated \emph{$d$-dimensional Bernoulli shift} is the shift dynamical system
$\big(A^{\Zd},(T^j)_{j\in\Zd},\varrho^{\otimes \Zd}\big)$.
\end{definition}


In concrete applications, it is natural to seek quantitative control of the coding radius,
for instance tail bounds for $r\!_\varphi$, or equivalently moment bounds for the coding volume
$|B_\infty(0,r\!_\varphi)|$.
We say that $\varphi$ has an \emph{integrable coding volume} if
$\int |B_\infty(0,r\!_\varphi(x))|\,\dd\mu(x) < \infty$,
which can be compactly written as $\E\big[\,|B_\infty(0,r\!_\varphi(X))|\,\big] < \infty$.
In many examples, one can even obtain exponential or stretched-exponential tail estimates, which implies that all moments of $|B_\infty(0,r\!_\varphi(X))|$ are finite.

\begin{remark}
More generally, one could work with random fields $(X_i)_{i\in\Zd}$ defined on an
arbitrary probability space.
All notions (coding radius, finitary coding, integrable radius/volume) extend verbatim to that
setting.
However, since our applications involve only invariant measures on the configuration spaces
$A^{\Zd}$ and $B^{\Zd}$, we formulate everything using the canonical representations for the sake of simplicity.
\end{remark}

We conclude this section with a brief caution regarding terminology and the distinction between ergodic-theoretic and dynamical notions of ergodicity.
When we speak of ergodicity of a shift-invariant probability measure (or random field), we always mean ergodicity in the sense of ergodic theory: a shift-invariant measure $\mu$ on $B^{\mathbb Z^d}$ is ergodic if every shift-invariant measurable set has $\mu$-measure $0$ or $1$, or equivalently if $\mu$ is an extreme point of the convex set of shift-invariant measures. This notion should not be confused with the use of \emph{ergodicity} for Markov chains or probabilistic cellular automata, where it typically refers to irreducibility and convergence to a unique invariant measure of the dynamics, possibly with quantitative rates.


\subsection{Gaussian concentration bounds}\label{sec:GCB}

For $j\in\Z^d$ and a measurable function $f:B^{\Z^d}\to\R$, define the (per-site) oscillation
\[
\delta_j f\;:=\;\sup\bigl\{|f(y)-f(y')|:\ y_\ell=y'_\ell,\ \forall\,\ell\neq j\bigr\}\in[0,\infty].
\]
The \emph{dependence set} of $f$ is
\[
\dep(f)\;:=\;\{\,i\in\Z^d:\ \delta_i f>0\,\}.
\]
We say that $f$ is \emph{local} if $\dep(f)$ is finite (written $\dep(f)\Subset\Z^d$).
Note that, by definition, $\dep(f)=\{i:\delta_i f>0\}$ is the smallest subset $\Lambda\subset\Z^d$ such that $f$ depends
only on the coordinates in $\Lambda$.

For an integer $p\geq 1$, let $\|\delta f\|_{p}:=\Big(\sum_{i\in\Zd} (\delta_i f)^p\Big)^{1/p}$. 

A local function $f$ has the bounded-difference property, or is said to be separately bounded,  if
\[
\delta_j f<+\infty,\;\forall j\in \Zd.
\]
Of course, for $j\notin\dep(f)$ we have $\delta_j f=0$. Bounded local functions obviously have the bounded-difference property.
Quasilocal functions are defined as uniform limits of local functions.

\begin{remark}
If $B$ is finite, then $B^{\Zd}$ is compact in the product topology, and quasilocal functions are exactly the continuous functions on $B^{\Zd}$ (which are bounded).
\end{remark}

We define the Gaussian concentration property for a random field.

\begin{definition}[Gaussian concentration]
Let $d\geq 1$ and $Y=(Y_i)_{i\in\Zd}$ be a $B$-valued random field where $B$ is a standard Borel space. Then $Y$ satisfies a Gaussian concentration bound if there exists a constant $C>0$ such
that, for any local function $f:B^{\Zd}\to\R$ with the bounded-difference property, and for any $\lambda>0$,  one has
\begin{equation}\label{def-GCB}
\log\E\big(\e^{\lambda (f(Y)-\E[f(Y)])}\big)\leq \frac{\sC}{2} \lambda^2 \|\delta f\|_2^2\,.    
\end{equation}
\end{definition}
If we are given the law $\nu$ of the random field $Y$ that satisfies \eqref{def-GCB}, we will simply say that $\nu$ satisfies Gaussian concentration.

Thus, a Gaussian concentration bound provides a specific type of upper bound on the cumulant moment generating function of the random variable $f(Y) - \E[f(Y)]$.
Note that since $\lambda f(Y)=(-\lambda)(-f(Y))$, it follows immediately that \eqref{def-GCB} also holds for any $\lambda<0$ .

A key feature of \eqref{def-GCB} is that the constant $C$ depends only on the underlying random field, not on the observable $f$; in particular, it is independent of $|\dep(f)|$, the size of the dependence set
(the sole $f$-dependence enters through $\|\delta f\|_2^2=\sum_{i\in\dep(f)} (\delta_i f)^2$).

By a standard argument (see {\em e.g.} \cite[Proposition 3.1]{CCR2017}), \eqref{def-GCB} implies the tail bounds
\begin{equation}\label{Gaussian-tails}
\Proba(\vert f(Y)-\E[f(Y)]\vert > u)
\leq 2 \exp\bigg(-\frac{u^2}{2C\|\delta f\|_2^2}\bigg), \quad \forall u>0.
\end{equation}
\begin{remark}
Conversely, if we assume that a random field $Y$ satisfies \eqref{Gaussian-tails} (for all local functions $f$ with the bounded-difference property and for $C>0$ independent of $f$), the reader can verify that \eqref{def-GCB} also holds, with a modified constant replacing $C$. We omit the details here.
Therefore, the Gaussian concentration bounds can equivalently be characterized by \eqref{def-GCB} or \eqref{Gaussian-tails}.
\end{remark}

Observe that shift invariance is not required in the definition of Gaussian concentration. However, in the sequel we will be interested only in shift-invariant measures. If a shift-invariant measure satisfies Gaussian concentration, one can show that it must be ergodic and, in fact, mixing in the ergodic-theoretic sense. We will see later that Gaussian concentration in fact forces an even stronger property, namely Bernoullicity.

\begin{remark}
An alternative terminology for \eqref{def-GCB} is to say that $Y$ is sub-Gaussian with variance proxy $C\,\|\delta f\|_2^2$, see, e.g., \cite{boucheron2013concentration,vershynin2018high,wainwright2019high}. 
\end{remark}

\begin{remark}[McDiarmid's inequality / i.i.d random variables]\label{rem:McDiarmid}
When \(Y\) is an i.i.d.\ random field, one can take \(C=1/8\) in \eqref{def-GCB}; this is McDiarmid's
inequality, also simply called the bounded differences inequality (see, e.g., \cite[Thm.~6.2, p.~171]{boucheron2013concentration}).
\end{remark}

Gaussian concentration has been established in a wide range of settings, including Markov chains, mixing processes, stochastic chains with unbounded memory, and Gibbs random fields, and it has found numerous applications, notably in mathematical statistics and in information theory. Even in the classical setting of independent random variables, its consequences are already striking, as it allows one to control fluctuations of observables that may be highly nonlinear or defined only implicitly. A non-exhaustive list of references includes \cite{boucheron2013concentration, chazottes/collet/kulske/redig/2007, ChazottesColletRedig2017, ChazottesColletRedig2026, chazottes2023gaussian, dedecker/gouezel/2015, DoucMoulinesPriouretSoulier2018, dubhashi2009concentration, kontorovich2008concentration, KontorovichRaginsky2017, kulske2003, marton1996measure, marton/1998, samson2000, vershynin2018high, wainwright2019high}.
In the context of dynamical systems, Gaussian concentration has also been proved for certain classes of nonuniformly hyperbolic systems, see for instance \cite{chazottes2012optimal}. In that setting, the notion of local oscillation is naturally replaced by partial Lipschitz constants, reflecting the geometric structure of the dynamics.

\subsection{Structural consequences for finite-valued random fields}\label{subsec:fvrf}

We restrict here to \emph{finite-valued} random fields, i.e., $B$-valued fields with $B$ finite. This class is already quite rich: it includes, in particular, many classical Gibbs random fields such as the Ising model. In this setting, we highlight two key consequences of Gaussian concentration. It implies that any such random field is Bernoulli. It also entails the positive relative entropy property, a known result that will play an important role in our applications to Gibbs measures.

\paragraph{Gaussian concentration implies Bernoullicity}

While not one of our main results, the following theorem shows that Gaussian concentration implies isomorphism to a Bernoulli shift, an important observation that, to the best of our knowledge, has not been explicitly noted before.

\begin{definition}[Bernoullicity]
Let $(B^{\Zd},(S^j)_{j\in\Z^d},\nu)$ be a measure-preserving shift dynamical system.
We say it is \emph{Bernoulli} if it is measure-theoretically isomorphic to a $d$-dimensional Bernoulli shift. 
\end{definition}
A measure-theoretic isomorphism is a coding map that is invertible modulo null sets: after removing sets of measure zero in the source and target, it becomes a bijection with a measurable inverse.
Since $B$ is finite in this section, the target Bernoulli shift can also be taken to be $B$-valued.

\begin{theo}[Gaussian concentration implies Bernoullicity]
\label{thm:gcb-implies-Bernoulli}\leavevmode\\
Let $Y=(Y_i)_{i\in\Z^d}$ be a $B$-valued random field whose law $\nu$ is ergodic for the shifts, and assume that $\nu$ satisfies Gaussian concentration. Then $\big(B^{\Z^d},(S^j)_{j\in\Z^d},\nu\big)$ is Bernoulli.
\end{theo}

\begin{proof}
The argument proceeds through the blowing-up property. If $Y$ has this property (see below), then it satisfies in particular the almost blowing-up property, which is known to be equivalent to being a coding of an i.i.d.\ random field; see \cite[Chs. III--IV]{shields/1996} for $d=1$, and note that the same argument applies for all $d\ge1$. 

It follows that $Y$ is a factor of a $d$-dimensional Bernoulli shift. By Ornstein's isomorphism theory for amenable group actions, any such factor is itself Bernoulli; see \cite{OrnsteinWeiss1987Amenable}. 

Thus, it remains only to show that Gaussian concentration implies the blowing-up property, which was established in \cite{CMRU} (in fact in a stronger quantitative form). This completes the proof.
\end{proof}

Bernoullicity admits several equivalent characterizations. In particular, it is equivalent to finite determination \cite{OrnsteinWeiss1987Amenable}. This means that for every $\varepsilon>0$ there exists a finite set $\Lambda\subset\mathbb{Z}^d$ such that, for any two stationary random fields on $B^{\mathbb{Z}^d}$ whose $\Lambda$-marginals are $\varepsilon$-close in total variation and whose entropy densities are $\varepsilon$-close, their $\bar d$-distance\footnote{The $\bar d$ (Ornstein) distance between two shift-invariant random fields is defined as the infimum, over all shift-invariant couplings (or joinings) of the fields, of the probability that the two configurations differ at the origin.} is at most $\varepsilon$.
This formulation highlights that Bernoullicity is a strong quantitative mixing property.

Let us briefly comment on the blowing-up property, introduced in information theory, which plays a central role in this connection. It was established by Marton and Shields \cite{marton/shields/1994} for finite-valued processes in dimension $d=1$, and extends without difficulty to finite-valued random fields.
Let $B$ be finite and $\Lambda\Subset\Zd$. For $x,x'\in B^\Lambda$, define the Hamming distance $\db_\Lambda(x,x')=\sum_{i\in\Lambda}\1_{\{x_i\neq x'_i\}}$. For $E\subseteq B^\Lambda$, set $\db_\Lambda(x,E)=\inf_{x'\in E}\db_\Lambda(x,x')$, and for $\varepsilon\in[0,1]$ define the $\varepsilon$-blowup $[E]_\varepsilon=\{x:\db_\Lambda(x,E)<\varepsilon|\Lambda|\}$.
An ergodic probability measure $\nu$ on $B^{\Zd}$ is said to have the blowing-up property if for every $\varepsilon>0$ there exist $\delta>0$ and $N$ such that for all $n\ge N$ and all $E\subseteq B^{\Lambda_n}$,
\[
\nu(E)\ge \e^{-(2n+1)^d\delta}\ \Longrightarrow\ \nu([E]_\varepsilon)\ge 1-\varepsilon,
\]
where $\nu(E)$ denotes $\nu(\{x: x_\Lambda\in E\})$.

We say that a random field has the blowing-up property if its law does. Moreover, this property is stable under finitary codings: if $Y$ is a finitary coding of an ergodic field $X$ with the blowing-up property, then $Y$ also has it; in particular, this holds when $X$ is i.i.d.

\begin{remark}
As mentioned in the proof of Theorem \ref{thm:gcb-implies-Bernoulli}, Gaussian concentration implies a quantitative form of the blowing-up property. We will present an example of a system that satisfies the blowing-up property but does not exhibit Gaussian concentration.
\end{remark}

\paragraph{The positive relative entropy property}

Given two shift-invariant probability measures $\nu,\nu'$ on $B^{\Zd}$, the lower relative entropy of $\nu'$ with respect to $\nu$ is defined by
\[
\ent_*(\nu'|\nu)
=
\liminf_{k\to\infty} \frac{1}{(2k+1)^d} 
\sum_{b\in B^{\{-k,\dots,k\}^d}} \nu'_k(b) \log\frac{\nu'_k(b)}{\nu_k(b)}\,,
\]
where $B^{\{-k,\dots,k\}^d}$ denotes the set of configurations indexed by $\{-k,\dots,k\}^d$, and $\nu_k$, $\nu'_k$ are the corresponding marginals of $\nu$ and $\nu'$, respectively.

One can likewise define the upper relative entropy $\ent^*(\nu'|\nu)$ by taking a limit superior. In general the lower and upper relative entropies need not coincide (pathologies can occur; see, for example, \cite{Shields1993} for $d=1$). However, in the context of Gibbs random fields, this is a well-behaved object.

\begin{definition}[Positive relative entropy property]
Let $Y=(Y_i)_{i\in\Z^d}$ be a $B$-valued random field with ergodic law $\nu$. We say that $Y$ has the \emph{positive relative entropy property} if
\[
\ent_*(\nu'|\nu)>0\quad\text{for every ergodic }\nu'\neq\nu.
\]
\end{definition}

We have the following result.

\begin{theo}[\cite{Chazottes/Redig/2022}]\label{theo-GCB-PREP}\leavevmode
Let $Y=(Y_i)_{i\in\Z^d}$ be a $B$-valued random field with ergodic law $\nu$, and assume that $\nu$ satisfies Gaussian concentration.
Then $Y$ has the positive relative entropy property.
\end{theo}

\begin{remark}
Once again, the blowing-up property plays a central role. Indeed, if\, $Y=(Y_i)_{i\in\Z^d}$ is a $B$-valued random field with ergodic law $\nu$ and has the blowing-up property, then it satisfies the positive relative entropy property. This was proved in \cite{marton/shields/1994} for $d=1$, and the argument extends readily to all $d\ge1$ (see \cite{CMRU}). Since Gaussian concentration implies the blowing-up property, the theorem follows.

We note, however, that \cite{Chazottes/Redig/2022} follows a different route, bypassing the blowing-up property and applying beyond the case of finite $B$, and in fact provides a quantitative strengthening by lower bounding $\ent_*(\cdot|\cdot)$ in terms of the square of the $\bar d$-distance.
\end{remark}

The interest of Theorem~\ref{theo-GCB-PREP} is that it can be used, in the context of phase transitions for Gibbs random fields, to show that being a coding of an i.i.d.\ process does not imply Gaussian concentration.

\section{Gaussian concentration for finitary codings of i.i.d. random fields}\label{sec:main}

We establish two main abstract theorems.
The first, Theorem \ref{main-thm-ffiid}, applies to general finitary codings and unavoidably requires finiteness of the second moment of the coding volume. This reflects the correlations inherently introduced by such codings. The necessity of this second-moment condition already emerges from the proof itself, and in Section \ref{sec:sharpness} we further explain why the result is sharp at this level of generality, in the absence of any additional structural assumptions.

Our second main result, Theorem \ref{theo:cone}, shows that under an additional abstract assumption, finiteness of the first moment of the coding volume is sufficient to obtain Gaussian concentration. This assumption is satisfied by random fields that can be simulated via a coupling-from-the-past algorithm, which so far constitutes the main general method for constructing finitary codings.

This first-moment condition is also sharp: there exist examples in which Gaussian concentration fails whenever the expected coding volume is infinite for every finitary coding. In particular, the finiteness of the first moment cannot be relaxed. In Section \ref{sec:sharpness} we also provide a more conceptual explanation of this obstruction.

\subsection{Finite second-moment coding volume implies Gaussian concentration}

We begin with a fully abstract result showing that Gaussian concentration is preserved under arbitrary finitary codings of i.i.d.\ random fields, provided the coding volume has a finite second moment.

\begin{theo}\label{main-thm-ffiid}
Let $d\ge 1$. 
Let $X=(X_i)_{i\in\Zd}$ be an i.i.d.\ $A$-valued random field, where $A$ is a standard Borel space, 
and let $Y=\varphi(X)$ for some finitary coding $\varphi:A^{\Zd}\to B^{\Zd}$. 
Assume the coding volume has a finite second moment, i.e.
\begin{equation}\label{finite-second-moment}
\E\left[\,|B_\infty(0,r\!_\varphi(X))|^2\,\right] < \infty\,.
\end{equation}
Then, for every local and continuous $f:B^{\Zd}\to\R$ with the bounded-difference property,
\[
\log \E\big[\exp\{\lambda(f(Y)-\E f(Y))\}\big]
\;\le\;
2^{d}\lambda^2\, \E\big[ (2r\!_\varphi(X)+1)^{2d}\,\big]\,
\|\delta f\|_2^2\,,
\qquad \forall\, \lambda>0.
\]
\end{theo}

\begin{remark}
When $B$ comes with the discrete topology, for instance when $B$ is finite, local functions are automatically continuous
(and the bounded-difference property is automatic as well).
\end{remark}

The proof of Theorem~\ref{main-thm-ffiid} relies on the following inequality, originally due to Talagrand and subsequently sharpened by Marton via a conditional transportation inequality. This result is also known as the bounded-differences inequality in quadratic mean; see \cite[Th.~8.6, p.~245]{boucheron2013concentration} and \cite[Chapter~4]{vH} for further details.

\begin{theo}[Marton's Gaussian concentration bound]
\label{thm:Marton}
\leavevmode\\
Let $X=(X_i)_{i\in\Zd}$ be an i.i.d. random field where the $X_i$ take values in a standard Borel space $A$. 
Let $g:A^{\Zd}\to\mathbb{R}$ be a local function.
Assume there are measurable functions $c_i:A^{\dep(g)}\to [0,+\infty)$, $i\in\dep(g)$, such that for all $x,x'\in A^{\Zd}$
\begin{equation}\label{generalized-bounded-difference-property}
\big|\, g(x)-g(x')\big|\leq \sum_{i\,\in\dep(g)} c_i(x)\, \1_{\{x_i\neq x'_i\}}.
\end{equation}
Then  
\begin{equation}\label{Marton-GCB}
\log\E\!\left[\e^{\lambda (g(X)-\E[g(X)])}\right]\leq 
\frac{\lambda^2}{2} \sum_{i\in\dep(g)} \E\big( c_i^2(X)\big),
\quad\forall \lambda>0.    
\end{equation}
\end{theo}

Theorem~\ref{thm:Marton} is a major upgrade over McDiarmid’s inequality, because instead of requiring deterministic worst-case Lipschitz constants, it allows the single-site sensitivity to depend on the configuration. This flexibility is precisely what is needed in our setting, where the coding radius is random and configuration dependent.

Our goal is to apply Theorem~\ref{thm:Marton} to observables of the form
\[
g = f\circ \varphi,
\]
where $\varphi$ is a finitary coding and $f$ is local. 
However, the map $\varphi$ has a random coding radius, and therefore the influence of a single input site on $g$ is itself random and a priori unbounded. 
To bring the situation within the scope of \eqref{generalized-bounded-difference-property}, we first truncate the coding map so as to obtain deterministic locality.

\begin{definition}[Truncation of the coding map]\label{def:truncA-d}
Let $\varphi:A^{\Zd}\to B^{\Zd}$ be a coding map, and $T^j$ the shift on $A^{\Zd}$.
Fix $n\in\mathbb{N}$ and a reference symbol $b_0\in B$. 
Define the truncated radius at site $j\in\Zd$ by
\[
r^{(n)}_{\varphi}(T^j x):=\min\{r_\varphi(T^j x),\,n\}.
\]
Define $\varphi^{(n)}:A^{\Zd}\to B^{\Zd}$ coordinatewise by
\[
\varphi^{(n)}(x)_j \;:=\;
\begin{cases}
\varphi(x)_j, & \text{if \;} r_\varphi(T^j x)\le n,\\[1mm]
b_0, & \text{if \;} r_\varphi(T^j x)>n.
\end{cases}
\]
\end{definition}

\begin{lemma}\label{lem:trunc-block}
$\varphi^{(n)}$ is measurable, shift-commuting, and is a block code of deterministic $\ell^\infty$-radius $\le n$, i.e.\ each coordinate $\varphi^{(n)}(x)_j$ depends only on $x$ restricted to $B_\infty(j,n)$.
Moreover, if $r\!_\varphi(T^j x)\le n$ then $\varphi^{(n)}(x)_j=\varphi(x)_j$.
\end{lemma}

\begin{proof}
If $r\!_\varphi(T^j x)\le n$, then $\varphi(x)_j$ is determined by $x$ on $B_\infty(j,n)$. If $r\!_\varphi(T^j x)>n$, then $\varphi^{(n)}(x)_j=b_0$, which is constant.
Hence each coordinate is a function of $x|_{B_\infty(j,n)}$. Shift-commutation follows from the definition; the last claim is by construction.
\end{proof}

Before giving the proof of Theorem \ref{main-thm-ffiid}, we prove a lemma. The key point is that we first telescope in the \emph{output coordinates}, and only then estimate the resulting indicators in terms of input disagreements. More precisely, writing $y=\varphi^{(n)}(x)$ and $y'=\varphi^{(n)}(x')$, locality of $f$ gives
\[
|f(y)-f(y')|
\le
\sum_{j\in\dep(f)} \delta_j f\,\1_{\{y_j\neq y'_j\}}.
\]
We then control each indicator $\1_{\{y_j\neq y'_j\}}$ using the deterministic block-code property of $\varphi^{(n)}$, which localizes the dependence of the $j$-th output coordinate to a fixed $\ell^\infty$-ball determined by the truncated coding radius at $x$. This yields a bound in which all influence coefficients depend only on the base configuration $x$, while the dependence on $x'$ appears solely through the indicators $\1_{\{x_i\neq x'_i\}}$.

\begin{lemma}\label{lem:multisite-d}
Let $f:B^{\Zd}\to\mathbb{R}$ be local. Let $n\in\N$, and let $\varphi^{(n)}$ be the truncated coding map (Definition \ref{def:truncA-d}). Then, for any $x,x'\in A^{\Zd}$,
\[
\big| f\circ \varphi^{(n)}(x)-f\circ \varphi^{(n)}(x') \big|
\le
\sum_{i\in\Zd}
\Bigg(
\sum_{j\in\dep(f)}
\delta_j f\;
\1\!\left\{ \|j-i\|_\infty\le r^{(n)}_\varphi(T^j x)\right\}
\Bigg)
\1_{\{x_i\neq x'_i\}}.
\]
\end{lemma}

\begin{proof}
Set $g^{(n)}:=f\circ\varphi^{(n)}$.
Write $y:=\varphi^{(n)}(x)$ and $y':=\varphi^{(n)}(x')$.
By telescoping over the output coordinates in $\dep(f)$,
\begin{equation}\label{eq:telescoping}
\big|g^{(n)}(x)-g^{(n)}(x')\big|
=
|f(y)-f(y')|
\le
\sum_{j\in\dep(f)} \delta_j f\;\1_{\{y_j\neq y'_j\}}.
\end{equation}
Fix $j\in\dep(f)$.
Since $\varphi^{(n)}$ is a block code of deterministic $\ell^\infty$-radius $\le n$ (Lemma~\ref{lem:trunc-block}),
each coordinate $\varphi^{(n)}(x)_j$ depends only on the restriction of $x$ to $B_\infty(j,n)$.
Hence, if $x$ and $x'$ agree on
\[
B_\infty\big(j,r_\varphi^{(n)}(T^j x)\big)
\subset B_\infty(j,n),
\]
then necessarily
\[
\varphi^{(n)}(x)_j=\varphi^{(n)}(x')_j.
\]
Equivalently,
\[
\1_{\{\varphi^{(n)}(x)_j\neq \varphi^{(n)}(x')_j\}}
\le
\sum_{i\in B_\infty(j,r_\varphi^{(n)}(T^j x))}
\1_{\{x_i\neq x'_i\}}.
\]
Injecting this bound into \eqref{eq:telescoping} yields
\begin{align*}
\big|g^{(n)}(x)-g^{(n)}(x')\big|
&\le
\sum_{j\in\dep(f)} \delta_j f
\sum_{i\in B_\infty(j,r^{(n)}_\varphi(T^j x))} \1_{\{x_i\neq x'_i\}}\\
&=
\sum_{i\in\Zd} \1_{\{x_i\neq x'_i\}}
\Bigg(
\sum_{j\in\dep(f)} \delta_j f\;
\1\!\left\{\|j-i\|_\infty\le r^{(n)}_\varphi(T^j x)\right\}
\Bigg),
\end{align*}
which is the desired bound.
\end{proof}

We introduce shorthand notation.
\begin{notation}
Since the coding map $\varphi$ is fixed throughout, we simplify the notation by writing $r^{(n)}_j(x)$ for $r^{(n)}_\varphi(T^j x)$, where $n\in\N$ and $j\in\Zd$.
\end{notation}

In view of Theorem~\ref{thm:Marton}, the structural bound of Lemma~\ref{lem:multisite-d} reduces Gaussian concentration for $g^{(n)}$ to the control of $\E \sum_{i\in\dep(g^{(n)})} \big(c_i^{(n)}(X)\big)^2$ where 
\[
c_i^{(n)}(x)\,:=\,
\sum_{j\in\dep(f)} \delta_j f\;
\1\!\left\{\|j-i\|_\infty\le r^{(n)}_\varphi(T^j x)\right\}.
\]
The next proposition shows that this expectation can be estimated in terms of the squared oscillations of $f$ and a purely coding-dependent convolution term involving the truncated radii.

\begin{proposition} \label{prop:upper}
Under the same condition as in Theorem \ref{main-thm-ffiid}, we have
\begin{equation*}
\E \sum_{i\in\dep(g^{(n)})} \big(c_i^{(n)}(X)\big)^2
\leq \|\delta f\|_2^2\, \|b\|_1,
\end{equation*}
where for all $j \in \Z^d$, 
\begin{equation} \label{def-b_j}
b_j =  \E\sum_{i\in\Z^d} \1\!\left\{\|i\|_\infty\le r_0^{(n)}(X)\right\} \1\!\left\{\|j-i\|_\infty\le r_j^{(n)}(X)\right\}\,.
\end{equation}
\end{proposition}

\begin{proof}
Let $f$ be local with the bounded-difference property and fix $n\in\N$.
Then
\begin{align*}
&\E \sum_{i\in\dep(g^{(n)})} \big(c_i^{(n)}(X)\big)^2\\
&\quad = \sum_{k,\ell\in\Z^d} \delta_k f\,\delta_\ell f\;
\E\sum_{i\in\Z^d}
\1\!\left\{\|k-i\|_\infty\le r_k^{(n)}(X)\right\}
\1\!\left\{\|\ell-i\|_\infty\le r_\ell^{(n)}(X)\right\}\\
&\quad= \sum_{k,\ell\in\Z^d} \delta_k f\,\delta_\ell f\; b_{\ell,k}\,,
\end{align*}
where 
\[
b_{\ell,k}:=\E\sum_{i\in\Z^d}\1\!\left\{\|k-i\|_\infty\le r_k^{(n)}(X)\right\}\1\!\left\{\|\ell-i\|_\infty\le r_\ell^{(n)}(X)\right\}\,.
\]
By shift invariance of $X$ we have $b_{\ell,k}=b_{\ell-k,0}=b_{\ell-k}$, where $b_{\ell-k}$ is defined in \eqref{def-b_j}.
Hence, the quadratic form can be rewritten as a convolution,
\[
\sum_{k,\ell} \delta_k f\,\delta_\ell f\, b_{\ell-k}
=\sum_{k} \delta_k f\,(\delta f * b)_k
\;\le\; \|\delta f\|_p\,\|\delta f\|_q\,\|b\|_r, \quad \tfrac1p+\tfrac1q+\tfrac1r=2\,,
\]
by Young's inequality. With $p=q=2$ and $r=1$,
\begin{equation}\label{pluie}
\E\sum_{i\in\dep(g^{(n)})} \big(c_i^{(n)}(X)\big)^2
\;\le\; \|\delta f\|_2^2\,\|b\|_1.
\end{equation}
This concludes the proof of the proposition.
\end{proof}

We now turn to the proof of Theorem \ref{main-thm-ffiid}.

\begin{proof}[Proof of Theorem~\ref{main-thm-ffiid}]
Let $f$ be local and satisfy the bounded-difference property, and fix $n\in\N$. Assume that there exists a finitary coding $\varphi$ such that $Y=\varphi(X)$, where $X$ is an i.i.d. random field.
Set $g^{(n)}:=f\circ\varphi^{(n)}$. 
By Theorem \ref{thm:Marton}
\begin{equation}\label{cgf-of-f(Y)}
\log \E\big[\exp\{\lambda(g^{(n)}(X)-\E g^{(n)}(X))\}\big]
\;\le\;
\frac{\lambda^2}{2}\,\E\sum_{i\in\dep(g^{(n)})}\big(c_i^{(n)}(X)\big)^2,\;\forall \lambda>0\,.
\end{equation}
By Proposition \ref{prop:upper} 
\[
\E\sum_{i\in\dep(g^{(n)})} \big(c_i^{(n)}(X)\big)^2
\;\le\; \|\delta f\|_2^2\,\|b\|_1.
\]
We rewrite $b_{k}$ using $\ell^\infty$-balls:
\begin{align*}
b_k
&\quad =  \E\sum_{i\in\Z^d}
\1\!\left\{\|i\|_\infty\le r_0^{(n)}(X)\right\}
\1\!\left\{\|k-i\|_\infty\le r_k^{(n)}(X)\right\}\\
&\quad= \E\big|B_\infty(0,r_0^{(n)}(X))\cap B_\infty(k,r_k^{(n)}(X))\big|\,,
\end{align*}
since, by definition of $B_\infty(\cdot,\cdot)$, for each $i\in\Z^d$ we have the equivalences
\[
i\in B_\infty(0,r_0^{(n)}(X)) \iff \|i\|_\infty\le r_0^{(n)}(X),\;
i\in B_\infty(k,r_k^{(n)}(X)) \iff \|k-i\|_\infty\le r_k^{(n)}(X).
\]
We next bound $\|b\|_1$. Clearly,
\[
b_0=\E\big[|B_\infty(0,r_0^{(n)}(X))|\big]=\sum_{n\geq 0}\Proba\big(r_0^{(n)}(X)> n\big)\,.
\]
For $k\neq0$, we write
\begin{align*}
b_k
&=\E\big[|B_\infty(0,r_0^{(n)}(X))\cap B_\infty(k,r_k^{(n)}(X))|\big]\\
&=\E\big[|B_\infty(0,r_0^{(n)}(X))\cap B_\infty(k,r_k^{(n)}(X))|
\ \1\{r_0^{(n)}(X)\ge \|k\|_\infty/2\}\big]\\
&\quad+\E\big[|B_\infty(0,r_0^{(n)}(X))\cap B_\infty(k,r_k^{(n)}(X))|
\ \1\{r_k^{(n)}(X)\ge \|k\|_\infty/2\}\big]\\
&\le \E\big[|B_\infty(0,r_0^{(n)}(X))|\,\1\{r_0^{(n)}(X)\ge \|k\|_\infty/2\}\big]
+\E\big[|B_\infty(k,r_k^{(n)}(X))|\,\1\{r_k^{(n)}(X)\ge \|k\|_\infty/2\}\big]\\
&\le 2\,\E\big[|B_\infty(0,r_0^{(n)}(X))|\,\1\{r_0^{(n)}(X)\ge \|k\|_\infty/2\}\big],
\end{align*}
where the last inequality follows from shift invariance.
Here we used the observation that if $r_0^{(n)}(X)<\|k\|_\infty/2$ and $r_k^{(n)}(X)<\|k\|_\infty/2$, then
\[
B_\infty\!\big(0,r_0^{(n)}(X)\big)\cap B_\infty\!\big(k,r_k^{(n)}(X)\big)=\varnothing.
\]

Summing over $k$ and applying Tonelli's theorem, we obtain
\begin{align*}
\|b\|_1
& \;\le\; 2 \,\E\left[ |B_\infty(0,r_0^{(n)}(X))|\,
\sum_{k\in\Z^d}\1\{\,r_0^{(n)}(X)\ge \|k\|_\infty/2\}\right] \\
& =2 \,\E\left[ |B_\infty(0,r_0^{(n)}(X))|\,|B_\infty(0,2r_0^{(n)}(X))|\right].
\end{align*}
Since $|B_\infty(0,2r)|=(4r+1)^d\le (2(2r+1))^d=2^d(2r+1)^d$, we deduce
\begin{equation}\label{pouic}
\|b\|_1 \;\le\; 2^{d+1}\, \E\big[(2r_0^{(n)}(X)+1)^{2d}\big].
\end{equation}
Combining \eqref{cgf-of-f(Y)}, \eqref{pluie}, and \eqref{pouic}, we obtain, for every $\lambda>0$,
\begin{equation}\label{paf}
\log \E\Big[\exp\{\lambda\,(g^{(n)}(X)-\E g^{(n)}(X))\}\Big]
\;\le\;
2^{d}\,\lambda^2\,\|\delta f\|_2^2\,\E\big[(2r_0^{(n)}(X)+1)^{2d}\big]\,.
\end{equation}
Since $r_0^{(n)}(X)\uparrow r\!_\varphi(X)$ almost surely as $n\to\infty$, the right-hand side of \eqref{paf} increases by monotone convergence to
\(2^{d}\lambda^{2}\,\|\delta f\|_{2}^{2}\,\E[(2r\!_\varphi(X)+1)^{2d}]\).
Moreover, since $\varphi^{(n)}(X)\to \varphi(X)=Y$ almost surely, we have \(g^{(n)}(X)\to f(Y)\) almost surely (by continuity of $f$). Applying dominated convergence twice, we conclude that
\[
\log \E\Big[\exp\{\lambda\,(f(Y)-\E f(Y))\}\Big]
\;\le\;
2^{d}\,\lambda^{2}\,\|\delta f\|_{2}^{2}\,\E\big[(2r\!_\varphi(X)+1)^{2d}\big]\,.
\]
We now use assumption \eqref{finite-second-moment}, which ensures that the right-hand side is finite; otherwise, the bound would be vacuous.
Since this holds for every $\lambda>0$ and every local continuous function $f$ with the bounded-difference property, we obtain the desired bound, which completes the proof of the theorem.
\end{proof}

\medskip

A natural question is whether the moment assumption on the coding volume in Theorem~\ref{main-thm-ffiid} can be relaxed. 
In particular, can one replace the second-moment requirement by the weaker condition of a finite mean, i.e.
\[
\E\big[\,|B_\infty(0,r\!_\varphi(X))|\,\big]\;<\;\infty\ ?
\]
This question is relevant, given that we can show that this condition is sharp and cannot, in general, be relaxed, see Proposition \ref{prop:Ising-at-betac} for a more explicit statement.

\begin{proposition}
There exists a random field that does not satisfy Gaussian concentration and for which no finitary coding by an i.i.d.\ random field can have finite expected coding volume.
\end{proposition}

In the next section, we show that, under an additional abstract assumption, there exists a class of finitary codings of i.i.d.\ random fields with finite expected coding volume that satisfy Gaussian concentration. We will see in Section \ref{sec:applications} that this assumption is met in all random-field examples studied so far.

\subsection{Gaussian Concentration with Finite First-Moment Coding Volume: A Sufficient Condition}\label{sec:cones}

We now introduce a condition under which the bound in Theorem \ref{main-thm-ffiid} can be improved by one moment.

\begin{definition}[Short-range factorization property]\label{def-short-range-factorization-property}
We say that a coding satisfies the short-range factorization property with constant $\alpha \in (0,1]$ if, for all $k,\ell,i  \in \Z^d$ such that $\max\{\|\ell-i\|,\|k-i\|, \|\ell-k\|\} = \|\ell-i\|$ 
the following holds
 \begin{align*}
&\E\Big[
\1\!\left\{\|k-i\|_\infty\le r_k(X)\right\}
\1\!\left\{\|\ell-k\|_\infty\le r_\ell(X)\right\}\Big] \\
&\leq \E\Big[
\1\!\left\{\|k-i\|_\infty\le r_k(X)\right\}\Big] \E\Big[
\1\!\left\{\alpha \|\ell-k\|_\infty\le r_\ell(X)\right\}\Big].
 \end{align*}
\end{definition}

In the following theorem, we observe that when the short-range factorization property holds, the term $\E\big[(2r_\varphi(X)+1)^{2d}\big]$ in the upper bound for the cumulant generating function can be replaced by $\big(\E\big[(2r_\varphi(X)+1)^{d}\big]\big)^2$.

\begin{theo}\label{theo:cone}
Let $d\ge 1$. 
Let $X=(X_i)_{i\in\mathbb Z^d}$ be an i.i.d.\ $A$-valued random field, where $A$ is a standard Borel space, 
and let $Y=\varphi(X)$ for some finitary coding $\varphi:A^{\mathbb Z^d}\to B^{\mathbb Z^d}$. 
If the coding satisfies the short-range factorization property with constant $\alpha\in(0,1]$, then for every local function
$f:B^{\mathbb Z^d}\to\R$ with the bounded-difference property,
\[
\log \E\big[\exp\{\lambda(f(Y)-\E f(Y))\}\big]
\;\le\;
3\,\alpha^{-d}\,\lambda^2\,
\big(\E\big[(2r_\varphi(X)+1)^{d}\big]\big)^2\,
\|\delta f\|_2^2,
\qquad \forall\, \lambda>0.
\]
\end{theo}

\begin{proof}
By Proposition \ref{prop:upper}, it suffices to bound, uniformly in $k\in\Z^d$,
\begin{equation}\label{eq:cone-target}
\sum_{\ell\in\Z^d}\sum_{i\in\Z^d}
\E\Big[
\1\!\left\{\|k-i\|_\infty\le r_k^{(n)}(X)\right\}
\1\!\left\{\|\ell-i\|_\infty\le r_\ell^{(n)}(X)\right\}
\Big].
\end{equation}
We decompose the sum according to which of the three distances
$\|k-i\|_\infty$, $\|\ell-i\|_\infty$, or $\|\ell-k\|_\infty$
is maximal. By symmetry and shitf invariance, it is enough to treat the case $k=0$.

\smallskip
\noindent Case 1: $\|\ell-i\|_\infty=\max\{\|i\|_\infty,\|\ell\|_\infty,\|\ell-i\|_\infty\}$.  
In this case,
\[
\1\!\left\{\|i\|_\infty\le r_0^{(n)}(X)\right\}
\1\!\left\{\|\ell-i\|_\infty\le r_\ell^{(n)}(X)\right\}
\le
\1\!\left\{\|i\|_\infty\le r_0^{(n)}(X)\right\}
\1\!\left\{\|\ell\|_\infty\le r_\ell^{(n)}(X)\right\}.
\]
Using the short-range factorization property and shift invariance,
\begin{align*}
&\E\Big[
\1\!\left\{\|i\|_\infty\le r_0^{(n)}(X)\right\}
\1\!\left\{\|\ell-i\|_\infty\le r_\ell^{(n)}(X)\right\}\Big] \\
&\le
\E\Big[\1\!\left\{\|i\|_\infty\le r_0^{(n)}(X)\right\}\Big]\;
\E\Big[\1\!\left\{\alpha\|\ell\|_\infty\le r_0^{(n)}(X)\right\}\Big].
\end{align*}
Summing over $i$ and $\ell$ yields a contribution bounded by
$\alpha^{-d}\big(\E[(2r_\varphi(X)+1)^d]\big)^2$.

\smallskip
\noindent Case 2: $\|i\|_\infty=\max\{\|i\|_\infty,\|\ell\|_\infty,\|\ell-i\|_\infty\}$.  We have
\begin{align*}
&\E\Big[
\1\!\left\{\|i\|_\infty\le r_0^{(n)}(X)\right\}
\1\!\left\{\|\ell-i\|_\infty\le r_\ell^{(n)}(X)\right\}\Big] \\
&\le
\E\Big[\1\!\left\{\alpha\|i\|_\infty\le r_0^{(n)}(X)\right\}\Big]\;
\E\Big[\1\!\left\{\|\ell-i\|_\infty\le r_0^{(n)}(X)\right\}\Big].
\end{align*}
Because, for all $i \in \Z^d$
\begin{equation*}
    \sum_{\ell \in \Z^d}\E\Big[\1\!\left\{\alpha\|\ell - i\|_\infty\le r_0^{(n)}(X)\right\}\Big] = \sum_{\ell \in \Z^d}\E\Big[\1\!\left\{\alpha\|\ell\|_\infty\le r_0^{(n)}(X)\right\}\Big],
\end{equation*}
we obtain the upper bound
$\alpha^{-d}\big(\E[(2r_\varphi(X)+1)^d]\big)^2$.

\smallskip
\noindent Case 3: $\|\ell\|_\infty=\max\{\|i\|_\infty,\|\ell\|_\infty,\|\ell-i\|_\infty\}$.  
Here we use the trivial bound
\begin{align*}
&\E\Big[
\1\!\left\{\|i\|_\infty\le r_0^{(n)}(X)\right\}
\1\!\left\{\|\ell-i\|_\infty\le r_\ell^{(n)}(X)\right\}\Big] \\
&\le
\E\Big[\1\!\left\{\|i\|_\infty\le r_0^{(n)}(X)\right\}\Big]\;
\E\Big[\1\!\left\{\|\ell-i\|_\infty\le r_0^{(n)}(X)\right\}\Big],
\end{align*}
which leads to a contribution bounded by $\big(\E[(2r_\varphi(X)+1)^d]\big)^2$.

\smallskip
Combining the three cases and using $\alpha\le1$, we conclude that
\[
\eqref{eq:cone-target}
\;\le\;
3\,\alpha^{-d}\,\big(\E\big[(2r_\varphi(X)+1)^d\big]\big)^2.
\]
This completes the proof.
\end{proof}

\begin{remark}
The numerical constant $3$ arises from a rough partition of the sum according to which of the three distances
$\|k-i\|_\infty$, $\|\ell-i\|_\infty$, or $\|\ell-k\|_\infty$ is maximal. This constant is not optimal and could be slightly improved by a more refined consideration of the summands in the decomposition. 
\end{remark}

\subsection{Sharpness of the moment conditions}\label{sec:sharpness}

We show that the dependence on the coding volume in
Theorems~\ref{main-thm-ffiid} and~\ref{theo:cone} is essentially optimal.

The only step in the proof where a genuine upper bound is used is
Proposition~\ref{prop:upper}, based on Young's inequality for discrete
convolutions. We first show that this bound is sharp.

The key mechanism is that oscillations of a local observable can be spread
over large regions so that many translated copies overlap. For block
functions, these overlaps are almost maximal, and the associated quadratic
form asymptotically reaches its $\ell^1$ norm.

\begin{proposition}[Optimality of the $\ell^1$ convolution bound]
Let $b=(b_m)_{m\in\mathbb Z^d}$ be a nonnegative function in $\ell^1(\mathbb Z^d)$, that is,
$b_m\ge 0$ for all $m\in\mathbb Z^d$ and $\sum_{m\in\mathbb Z^d} b_m < \infty$.
For any $\delta=(\delta_k)_{k\in\mathbb Z^d}\in \ell^2(\mathbb Z^d)$ with finite support, define
\[
Q(\delta) := \sum_{k,\ell\in\mathbb Z^d} \delta_k\,\delta_\ell\, b_{\ell-k}.
\]
Then
\[
\sup_{\delta\in \ell^2(\mathbb Z^d),\, \delta\neq 0}
\frac{Q(\delta)}{\|\delta\|_2^2}
= \|b\|_1.
\]
Moreover, if $\delta^{(L)}=\1_{\Lambda_L}$ with
$\Lambda_L := [-L,L]^d\cap\mathbb Z^d$, then
\[
\frac{Q(\delta^{(L)})}{\|\delta^{(L)}\|_2^2}
\longrightarrow
\|b\|_1
\qquad\text{as } L\to\infty.
\]
\end{proposition}

\begin{proof}
We write
\[
Q(\delta)=\sum_{k,\ell} \delta_k\,\delta_\ell\, b_{\ell-k}
= \sum_k \delta_k\, (b*\delta)_k
= \langle \delta, b*\delta\rangle.
\]
By Cauchy-Schwarz and Young's inequality,
\[
Q(\delta)
\le \|\delta\|_2\,\|b*\delta\|_2
\le \|b\|_1\,\|\delta\|_2^2,
\]
which gives the upper bound.

For $\delta^{(L)}=\1_{\Lambda_L}$, we compute
\[
Q(\delta^{(L)})
= \sum_{k,\ell} \1_{\Lambda_L}(k)\,\1_{\Lambda_L}(\ell)\, b_{\ell-k}
= \sum_m b_m\, |\Lambda_L\cap(\Lambda_L-m)|,
\]
so that
\[
\frac{Q(\delta^{(L)})}{\|\delta^{(L)}\|_2^2}
=
\sum_m b_m\,\frac{|\Lambda_L\cap(\Lambda_L-m)|}{|\Lambda_L|}.
\]
For each fixed $m$, one has
\[
\frac{|\Lambda_L\cap(\Lambda_L-m)|}{|\Lambda_L|}\longrightarrow 1
\qquad\text{as } L\to\infty,
\]
and the ratio is bounded by $1$. Since $b\in\ell^1(\mathbb Z^d)$, the claim follows by dominated convergence.
\end{proof}

Applying this to our setting, with the truncated coding map and coding radius introduced in Definition~\ref{def:truncA-d}, yields the following.

\begin{cor}
Fix \(n\in\mathbb N\), and let \(r^{(n)}_\varphi\) denote the truncated coding radius. Define
\[
b_m := \mathbb E\!\big[
|B_\infty(0,r^{(n)}_0)\cap B_\infty(m,r^{(n)}_{m})|
\big],
\qquad m\in\mathbb Z^d.
\]
For each \(L\ge 1\), let \(\Lambda_L := [-L,L]^d\cap\mathbb Z^d\), and let \(f^{(L)}\) be a local observable such that
\[
\delta_k f^{(L)} = 1 \quad \text{for all } k\in \Lambda_L,
\qquad
\delta_k f^{(L)} = 0 \quad \text{for } k\notin \Lambda_L.
\]
Then \(\|\delta f^{(L)}\|_2^2 = |\Lambda_L|\), and
\[
\liminf_{L\to\infty}
\frac{\mathbb E\sum_{i\in\mathbb Z^d}\big(c_i^{(n)}(X)\big)^2}
{\|\delta f^{(L)}\|_2^2}
=
\|b\|_1.
\]
\end{cor}

\noindent
This shows that the bound
\[
\mathbb E\sum_{i\in\mathbb Z^d}\big(c_i^{(n)}(X)\big)^2
\le
\|\delta f\|_2^2\,\|b\|_1
\]
is asymptotically sharp for block observables: when the oscillation is spread uniformly over a large region, the overlap structure of the truncated coding windows produces maximal reinforcement, and the quadratic form attains its $\ell^1$ norm.

In particular, in the setting of Theorem~\ref{main-thm-ffiid}, where $\|b\|_1$ is controlled by the second moment of the coding volume, this shows that the second-moment scale cannot be improved by analytic arguments alone.

We next show that no universal bound can depend on less than the first
moment of the coding volume.

\begin{proposition}[No universal bound below the first moment]
Let \(K>0\), and let \(\Psi\) be a nondecreasing functional on the class of nonnegative integer-valued random variables. Assume that for every finitary coding \(\varphi\), every i.i.d.\ input \(X\), every local observable \(f\), and every \(n\in\mathbb N\),
\[
\E\sum_{i\in\mathbb Z^d}\big(c_i^{(n)}(X)\big)^2
\le
K\,\|\delta f\|_2^2\,\Psi\big(r_\varphi^{(n)}(X)\big).
\]
Then necessarily
\[
\E\,|B_\infty(0,r_\varphi^{(n)}(X))|
\le
K\,\Psi\big(r_\varphi^{(n)}(X)\big).
\]
In particular, any universal squared-influence bound depending only on the coding radius must control at least the first moment of the coding volume.
\end{proposition}

\begin{proof}
Choose a single-site observable \(f\) depending only on the coordinate at the origin and normalized so that
\[
\delta_0 f=1.
\]
Then \(\dep(f)=\{0\}\), \(\delta_j f=0\) for \(j\neq 0\), and therefore
\[
\|\delta f\|_2^2=1.
\]
Moreover, by the definition of \(c_i^{(n)}\),
\[
c_i^{(n)}(x)
=
\sum_{j\in\dep(f)} \delta_j f\;
\1\!\left\{\|j-i\|_\infty\le r^{(n)}_\varphi(T^j x)\right\}
=
\1\!\left\{\|i\|_\infty\le r^{(n)}_\varphi(x)\right\}.
\]
Hence, pointwise,
\[
\sum_{i\in\mathbb Z^d}\big(c_i^{(n)}(x)\big)^2
=
\sum_{i\in\mathbb Z^d}\1\!\left\{\|i\|_\infty\le r^{(n)}_\varphi(x)\right\}
=
|B_\infty(0,r^{(n)}_\varphi(x))|.
\]
Taking expectations and applying the assumed bound yields
\[
\E\,|B_\infty(0,r_\varphi^{(n)}(X))|
\le
K\,\Psi\big(r_\varphi^{(n)}(X)\big),
\]
as claimed.
\end{proof}

These obstructions are consistent with concrete models. For instance, as we
shall see below, in the Ising model at criticality ($d\ge2$), every finitary
coding has infinite mean coding volume, and Gaussian concentration fails,
although a finitary coding still exists.

Taken together, these results show that the moment conditions in
Theorems~\ref{main-thm-ffiid} and~\ref{theo:cone} are optimal at two distinct
levels: the second moment arises from the geometry of overlaps, while the first
moment reflects a universal obstruction that cannot be bypassed without
additional structure.

\section{Applications and examples}\label{sec:applications}

In this section we illustrate the scope of the abstract results of Section~\ref{sec:main} through a range of examples from statistical mechanics, interacting particle systems, and stochastic processes. In each case, the strategy is the same: combine an existing finitary-coding construction with one of our abstract concentration theorems.

Our main applications concern Gibbs measures and Markov random fields on $\Z^d$, including the ferromagnetic Ising, Potts, and random-cluster models. Several approaches to Gaussian concentration are available in this setting, but they are rather heterogeneous. In the classical high-temperature regime, one may use Dobrushin’s uniqueness criterion \cite{kulske2003}, and, for finite-range interactions, disagreement-percolation methods provide another route \cite{chazottes/collet/kulske/redig/2007}. Alternatively, one can proceed via logarithmic Sobolev inequalities, which are known under suitable mixing conditions and are generally understood to imply Gaussian concentration through the Herbst argument, although this implication is not always stated explicitly in the lattice setting; recent work of Bauerschmidt and Dagallier \cite{bauerschmidt_dagallier_2024} establishes such inequalities for the Ising model throughout the uniqueness regime. These approaches, however, apply under different assumptions and do not yield a unified picture.

In this context, we recover known results in a unified framework and substantially extend them. Previous approaches based on Dobrushin-type or disagreement-percolation conditions were confined to strict subregimes of uniqueness. By contrast, our results apply throughout the full uniqueness regime, thereby covering models that were inaccessible to earlier techniques, and yield several new consequences.

Our approach builds on recent progress on finitary codings of Gibbs measures. For specific models such as the Ising, Potts, and random-cluster models, results of \cite{spinkaEJP, spinka-harel} provide finitary codings with good tail behavior, which, combined with our abstract results, yield Gaussian concentration together with sharp characterizations in terms of the phase diagram.

In addition, a general route is provided by spatial mixing: by combining our results with the construction of finitary couplings from the past under exponential strong spatial mixing from \cite{spinka-AoP}, we obtain Gaussian concentration for a broad class of models. This includes models where classical techniques based on Dobrushin-type conditions or disagreement percolation do not apply.

We also discuss a non-equilibrium example, namely the parking process, as well as one-dimensional processes, including both Markov chains and chains with unbounded memory, and more generally left-finitary processes, which extend these classes.

Taken together, these examples show that Gaussian concentration is robust under finitary codings with controlled coding volume, yet sensitive enough to detect qualitative changes in the underlying dependence structure.

\subsection{Gibbs measures and Markov random fields on \texorpdfstring{$\Z^d$}{Zd}}

Gaussian concentration was already known under Dobrushin's uniqueness condition \cite{kulske2003}, which in particular covers infinite-range interactions.\footnote{In that paper, $B$ may be a standard Borel space and $\Z^d$ may be any countable set.}
In \cite{chazottes/collet/kulske/redig/2007}, a coupling method was developed, yielding Gaussian concentration for finite-range interactions under van den Berg and Maes's disagreement-percolation criterion. Moreover, Dobrushin uniqueness, disagreement percolation, and H\"aggstr\"om-Steif's high-noise condition each imply exponential strong spatial mixing, not to be confused with ergodic-theoretic mixing.

A major advance in this direction was obtained by Spinka \cite{spinka-AoP}, who showed that finite-valued Markov random fields with exponential strong spatial mixing are finitary factors of i.i.d.\ random fields, with exponential or stretched-exponential tails for the coding radius. Combining this with Theorem~\ref{main-thm-ffiid} gives a unified route to Gaussian concentration: we recover previously known cases and obtain new ones. In particular, for the ferromagnetic Ising and Potts models, this approach yields necessary and sufficient conditions in terms of the inverse temperature. Using Harel and Spinka \cite{spinka-harel}, we also obtain new statements for certain monotone models of infinite range, including the random-cluster model.

We briefly recall the relevant Gibbsian formalism and fix notation; see \cite{GHM2001,georgii2011gibbs,friedli_velenik_2017,ruelle2004thermodynamic} for details. Many of the examples below are Markov random fields generated by nearest-neighbor or, more generally, finite-range interactions, though not all. We also allow hard constraints, so that the configuration space may be a proper subshift of the full shift.

To accommodate hard-core exclusions, we work on a subshift $\mathsf{Y}\subset B^{\Z^d}$, where $B$ is finite. Thus $\mathsf{Y}$ is a closed, shift-invariant subset of the full shift $(B^{\Z^d},(S^j)_{j\in\Z^d})$, interpreted as the set of feasible configurations. In many examples, $\mathsf{Y}$ is a subshift of finite type: the feasible configurations are precisely those in which no pattern from a fixed finite list of forbidden patterns occurs. When $\mathsf{Y}=B^{\Z^d}$ we recover the full shift. Coding maps and coding radii extend verbatim to subshifts.

An interaction is a family $\Phi=\{\Phi_\Lambda\}_{\Lambda\Subset\Z^d}$ of local functions with
\[
\Phi_\Lambda:\mathsf{Y}_\Lambda\to\R,
\qquad
\Phi_{\Lambda+i}=\Phi_\Lambda\circ S^i
\quad\text{for all }i\in\Z^d,
\]
where $\mathsf{Y}_\Lambda$ denotes the restriction of $\mathsf{Y}$ to $\Lambda$. The Hamiltonian in a finite box $\Lambda\Subset\Z^d$ is
\[
H_\Lambda(y)
:=
\sum_{\substack{\Lambda'\Subset\Z^d\\ \Lambda'\cap\Lambda\neq\emptyset}}
\Phi_{\Lambda'}(y_{\Lambda'}),
\qquad y\in\mathsf{Y}.
\]
Write
\[
\mathrm{range}(\Phi)
:=
\inf\bigl\{r>0:\Phi_\Lambda\equiv0\ \text{whenever }\mathrm{diam}(\Lambda)>r\bigr\},
\]
where $\mathrm{diam}$ is computed in the $\ell^1$-metric on $\Z^d$. If $\mathrm{range}(\Phi)<\infty$ we say that $\Phi$ has \emph{finite range}. If $\mathrm{range}(\Phi)=\infty$, we assume absolute summability:
\[
\|\Phi\|
:=
\sum_{\Lambda\Subset\Z^d:\ 0\in\Lambda}
\sup_{y_\Lambda\in\mathsf{Y}_\Lambda}|\Phi_\Lambda(y_\Lambda)|
<\infty.
\]

Given an interaction $\Phi$, a probability measure $\nu$ on $\mathsf{Y}$ is a \emph{Gibbs measure} if for every $\Lambda\Subset\Z^d$,
\[
\nu\big([y_\Lambda]\mid \mathfrak B_{\Lambda^\comp}\big)(y')
=
\frac{\exp\{-H_\Lambda(y_\Lambda y'_{\Lambda^\comp})\}}{Z_\Lambda^{y'}}
\quad\text{for }\nu\text{-a.e.\ }y',
\]
where $\mathfrak B_{\Delta}$ is the product sigma-field on $\mathsf{Y}_\Delta$,
\[
Z_\Lambda^{y'}
:=
\sum_{z_\Lambda\in\mathsf{Y}_\Lambda}
\exp\{-H_\Lambda(z_\Lambda y'_{\Lambda^\comp})\},
\]
and $[y_\Lambda]:=\{x\in B^{\Z^d}:x_\Lambda=y_\Lambda\}$ denotes the corresponding cylinder set. For absolutely summable interactions there exists at least one shift-invariant Gibbs measure. A Gibbs measure is called \emph{extremal} if it cannot be written as a nontrivial convex combination of other Gibbs measures; in the shift-invariant setting, this is equivalent to ergodicity. Typically $\Phi$ depends on parameters such as inverse temperature or fugacity. When $\nu$ is a Gibbs measure, it induces a Gibbs random field $Y=(Y_i)_{i\in\Z^d}$ with law $\nu$.

For $r\in\N$, define the $r$-boundary of a finite set $\Lambda\Subset\Z^d$ by
\[
\partial_r\Lambda
:=
\{i\in\Lambda^\comp:\mathrm{dist}(i,\Lambda)\le r\},
\]
where $\mathrm{dist}$ is computed in the $\ell^1$-metric. We write $\partial\Lambda$ for $\partial_1\Lambda$. A shift-invariant measure $\nu$ on $\mathsf{Y}$ is called an $r$-Markov random field if for every finite $\Lambda\Subset\Z^d$, the conditional law of $Y_\Lambda$ given the outside depends only on $Y_{\partial_r\Lambda}$. When $r=1$, we simply say that $\nu$ is a Markov random field. This corresponds to a nearest-neighbor interaction. For an $r$-Markov random field, $\mathsf{Y}=\operatorname{supp}(\nu)\subset B^{\Z^d}$ is necessarily a subshift of finite type.

We use the notation
\[
E:=\big\{\{i,j\}\subset\Z^d:\|i-j\|_1=1\big\}
\]
for the set of nearest-neighbor edges.

\subsubsection{The ferromagnetic nearest-neighbor Ising model}

Take $B=\{-1,+1\}$. The Hamiltonian for the ferromagnetic Ising model at inverse temperature $\beta>0$, with zero external field, is
\[
H_\Lambda(y_\Lambda y'_{\Lambda^\comp})
=
-\sum_{\substack{\{i,j\}\in E\\ \{i,j\}\subset\Lambda}}\beta\,y_i y_j
-\sum_{\substack{\{i,j\}\in E\\ i\in\Lambda,\ j\in\partial\Lambda}}\beta\,y_i y'_j.
\]

It is well known that there exists $\beta_c(d)\in(0,\infty)$ such that the Gibbs measure is unique for $\beta\le\beta_c(d)$, while for $\beta>\beta_c(d)$ there are multiple ergodic Gibbs measures. In dimension $d=2$, all Gibbs measures are shift-invariant and form a convex combination of two extremal measures, denoted $\nu_\beta^+$ and $\nu_\beta^-$. These are obtained as weak limits, as $\Lambda\uparrow\Z^2$, of finite-volume Gibbs measures with all-$+$ and all-$-$ boundary conditions, respectively. They are the only ergodic Gibbs measures in this setting. When $\nu_\beta^+=\nu_\beta^-$, we write $\nu_\beta$ for the common measure.

\begin{theo}\label{thm:Ising}
For the ferromagnetic nearest-neighbor Ising model in dimension $d\ge2$, Gaussian concentration holds in the uniqueness regime. In the phase coexistence regime $\beta\ge\beta_c(d)$, it fails for every shift-invariant ergodic Gibbs measure. More precisely, for $\beta<\beta_c(d)$, the unique Gibbs measure $\nu_\beta$ satisfies Gaussian concentration, whereas for $\beta\ge\beta_c(d)$ no shift-invariant ergodic Gibbs measure satisfies Gaussian concentration.
\end{theo}

\begin{proof}
If $\beta<\beta_c(d)$, the conclusion follows directly from Theorem~\ref{main-thm-ffiid} together with Theorem~1.1 of \cite{spinkaEJP}, which provides a finitary coding by an i.i.d.\ random field with exponential tails for the coding radius (or by a finite-valued i.i.d.\ field with stretched-exponential tails).

Assume next that $\beta>\beta_c(d)$. Then $\ent_*(\nu_\beta^- \mid \nu_\beta^+)=\ent^*(\nu_\beta^- \mid \nu_\beta^+)=0$; see \cite{georgii2011gibbs}. Hence the positive relative entropy property fails, and therefore Theorem~\ref{theo-GCB-PREP} implies that no shift-invariant ergodic Gibbs measure can satisfy Gaussian concentration.

Finally, consider the critical case $\beta=\beta_c(d)$. By \cite{aizenman2015random}, the ferromagnetic Ising model on $\Z^d$ admits a unique infinite-volume Gibbs measure at criticality; let $Y=(Y_k)_{k\in\Z^d}$ denote the corresponding Gibbs random field.

Suppose, for contradiction, that $Y$ satisfies Gaussian concentration. Then there exists $C<\infty$ such that for every local function $f:\{-1,+1\}^{\Z^d}\to\R$,
\begin{equation}\label{eq:Ising-var-bound}
\Var(f(Y))
\le
C\,\|\delta f\|_2^2.
\end{equation}
Indeed, apply the Gaussian concentration inequality to $\lambda f$, subtract $1$, divide by $\lambda^2$, and let $\lambda\to0$.

Let $\Lambda_n=B_\infty(0,n)$ and define
\[
S_n:=\sum_{k\in\Lambda_n}Y_k.
\]
At criticality, the Ising Gibbs state is centered and ferromagnetic, so that
\[
\E(Y_0)=0,
\qquad
\Cov(Y_0,Y_j)=\E(Y_0Y_j)\ge0
\quad\text{for all }j\in\Z^d.
\]
Moreover, the susceptibility diverges:
\begin{equation}\label{eq:Ising-susceptibility}
\sum_{k\in\Z^d}\E(Y_0Y_k)=+\infty.
\end{equation}
As a consequence,
\begin{equation}\label{eq:Ising-var-diverges}
\frac{1}{|\Lambda_n|}\Var(S_n)\xrightarrow[n\to\infty]{}+\infty.
\end{equation}
Indeed, by shift-invariance,
\begin{align*}
\frac{1}{|\Lambda_n|}\Var(S_n)
&=
\sum_{k\in\Z^d}
\frac{|\Lambda_n\cap(\Lambda_n-k)|}{|\Lambda_n|}
\,\E(Y_0Y_k).
\end{align*}
Hence, for every fixed $R\ge1$,
\[
\liminf_{n\to\infty}\frac{1}{|\Lambda_n|}\Var(S_n)
\ge
\sum_{\|k\|_\infty\le R}\E(Y_0Y_k),
\]
and letting $R\to\infty$ yields \eqref{eq:Ising-var-diverges}.

On the other hand, applying \eqref{eq:Ising-var-bound} to $f(Y)=S_n$ gives
\[
\Var(S_n)\le C\,\|\delta f\|_2^2\le 4C\,|\Lambda_n|,
\]
so that
\[
\limsup_{n\to\infty}\frac{1}{|\Lambda_n|}\Var(S_n)\le 4C,
\]
contradicting \eqref{eq:Ising-var-diverges}. This completes the proof.
\end{proof}

\begin{remark}
Recall that any shift-invariant measure satisfying Gaussian concentration is necessarily ergodic. 
When $d=2$, the only ergodic Gibbs measures of the ferromagnetic Ising model are $\nu_\beta^+$ and $\nu_\beta^-$. It follows that Gaussian concentration holds for $\beta<\beta_c$, while for $\beta>\beta_c$ it fails for both $\nu_\beta^+$ and $\nu_\beta^-$.

In contrast with the two-dimensional case, where all Gibbs measures are shift-invariant, this is no longer true in dimension $d=3$. At sufficiently low temperature, one encounters the so-called Dobrushin states, which are extremal but not shift-invariant, and therefore do not correspond to equilibrium states. We refer to \cite{georgii2011gibbs} for these results.

Although Gaussian concentration itself does not require shift invariance a priori, the theorem above is restricted to shift-invariant Gibbs measures, since our argument in the coexistence regime relies crucially on shift invariance. It therefore remains open whether certain non-shift-invariant Gibbs measures may satisfy Gaussian concentration. We do not expect this to hold for Dobrushin states, in view of the presence of macroscopic interface fluctuations.
\end{remark}

The next proposition shows that the critical Ising model realizes the optimal obstruction behind Theorem~\ref{theo:cone}. Although finitary codings from an i.i.d.\ random field do exist at criticality, every such coding must have infinite expected coding volume. Thus the failure of Gaussian concentration and the impossibility of finite expected coding volume have a common origin, namely the divergence of the susceptibility.

\begin{proposition}[Ising model at criticality]\label{prop:Ising-at-betac}
Let $d\ge2$. For the ferromagnetic Ising model at $\beta=\beta_c(d)$, the unique Gibbs measure does not satisfy Gaussian concentration. Nevertheless, it satisfies the blowing-up property, since it admits a finitary coding from an i.i.d.\ random field. Moreover, any finitary coding of this random field by an i.i.d.\ random field necessarily has infinite expected coding volume.
\end{proposition}

\begin{proof}
The failure of Gaussian concentration at criticality is established in Theorem~\ref{thm:Ising}. On the other hand, the Ising model at $\beta=\beta_c(d)$ admits a finitary coding from an i.i.d.\ random field; in particular, \cite{timar2025factor} constructs such a coding from a finite-valued i.i.d.\ source. Consequently, the corresponding Gibbs measure satisfies the blowing-up property; see Subsection~\ref{subsec:fvrf}. Finally, Theorem~4.3 of \cite{steif/vanderberg/1999} shows that at $\beta=\beta_c(d)$, the existence of a finitary coding already forces the expected coding volume to be infinite:
\[
\E\big[\,|B_\infty(0,r_\varphi(Y))|\,\big]=\infty.
\]
\end{proof}

\subsubsection{The random-cluster model}

In contrast with classical nearest-neighbor models such as the Ising model or the Potts model, the random cluster model is inherently non-local:
the conditional distribution of a single edge depends on the entire configuration through global connectivity properties.
In particular, it cannot be described by a finite-range
interaction.

Let $E$ again denote the set of nearest-neighbor edges of $\Z^d$. A configuration is an element $y\in\{0,1\}^{E}$, where $y(e)=1$ means that $e$ is open.

For parameters $p\in[0,1]$ and $q\ge 1$, the random-cluster model admits two standard infinite-volume Gibbs measures, the free and wired measures, denoted by $\phi^{\mathrm{free}}_{p,q}$ and $\phi^{\mathrm{wired}}_{p,q}$. They are obtained as weak limits of the corresponding finite-volume measures with free and wired boundary conditions. Both are shift-invariant and ergodic. When they coincide, we write $\phi_{p,q}$ for the common measure.

It is known that there exists a critical threshold $p_c(q)\in[0,1]$ such that for each of the boundary conditions $i\in\{\mathrm{free},\mathrm{wired}\}$,
\[
\phi^i_{p,q}\{\exists\ \text{an infinite cluster}\}
=
\begin{cases}
0, & p<p_c(q),\\[4pt]
1, & p>p_c(q).
\end{cases}
\]
When $q=1$, the model reduces to Bernoulli bond percolation, in which case the infinite-volume measure is unique for every $p$.

\begin{theo}\label{thm:RC}
Let $d\ge2$ and $q>1$.
If $p<p_c(q)$, then $\phi_{p,q}$ satisfies Gaussian concentration.
If $p>p_c(q)$, then neither $\phi^{\mathrm{free}}_{p,q}$ nor $\phi^{\mathrm{wired}}_{p,q}$ satisfies Gaussian concentration.
\end{theo}

\begin{proof}
If $p<p_c(q)$, Theorem~1.3 of \cite{spinka-harel} shows that the model is a finitary coding of a finite-valued i.i.d.\ random field with stretched-exponential tails for the coding radius. The conclusion therefore follows from Theorem~\ref{main-thm-ffiid}.

If $p>p_c(q)$, there is phase coexistence: the two distinct infinite-volume Gibbs measures $\phi^{\mathrm{free}}_{p,q}$ and $\phi^{\mathrm{wired}}_{p,q}$ have zero relative entropy with respect to one another. Theorem~\ref{theo-GCB-PREP} therefore implies that Gaussian concentration cannot hold for either of them.
\end{proof}

In dimension $d=2$, every Gibbs measure is a convex combination of the free and wired measures. In particular, the only shift-invariant ergodic Gibbs measures are $\phi^{\mathrm{free}}_{p,q}$ and $\phi^{\mathrm{wired}}_{p,q}$ when they are distinct. Thus, in dimension $d=2$, Theorem~\ref{thm:RC} gives a complete picture away from criticality.

\subsubsection{The ferromagnetic nearest-neighbor Potts model}

Fix an integer $q\ge2$ and let $B=\{1,\dots,q\}$. For a finite box $\Lambda\Subset\Z^d$, inverse temperature $\beta>0$, and $i\in\{0,1,\dots,q\}$, define the finite-volume Hamiltonian with all-$i$ boundary condition by
\[
H_\Lambda(y_\Lambda\, i_{\Lambda^\comp})
=
-\sum_{\substack{\{u,v\}\in E\\ \{u,v\}\subset\Lambda}}
\beta\,\1\{y_u=y_v\}
-\sum_{\substack{\{u,v\}\in E\\ u\in\Lambda,\ v\in\partial\Lambda}}
\beta\,\1\{y_u=i\},
\]
where the second sum is interpreted as $0$ when $i=0$. The corresponding infinite-volume Gibbs measures are denoted by $\nu_{\beta,q}^0,\nu_{\beta,q}^1,\dots,\nu_{\beta,q}^q$; they are obtained as weak limits and are shift-invariant and ergodic. The case $q=2$ reduces, up to the usual relabeling of spins, to the Ising model.

Set
\[
\beta_c(q):=-\log\bigl(1-p_c(q)\bigr),
\]
where $p_c(q)$ is the random-cluster critical parameter. If $\beta<\beta_c(q)$, then it is well known that the measures $\nu_{\beta,q}^0,\nu_{\beta,q}^1,\dots,\nu_{\beta,q}^q$ all coincide; we denote the common measure by $\nu_{\beta,q}$.

\begin{theo}\label{thm:Potts}
Let $d\ge2$ and $q\ge2$.

If $\beta<\beta_c(q)$, then the (unique) Gibbs measure
$\nu_{\beta,q}$ satisfies Gaussian concentration.

If $\beta>\beta_c(q)$, then none of the extremal
shift-invariant Gibbs measures
$\nu_{\beta,q}^1,\dots,\nu_{\beta,q}^q$
satisfies Gaussian concentration.
\end{theo}

\begin{proof}
If $\beta<\beta_c(q)$, then the Gibbs measure is unique.
By Theorem~1.3 of \cite{spinka-harel}, the subcritical random-cluster
model admits a finitary coding from a finite-valued i.i.d.\ process
with stretched-exponential coding-radius tails.
Via the Edwards–Sokal coupling, the same holds for the Potts model.
The claim then follows from Theorem~\ref{main-thm-ffiid}.

Assume $\beta>\beta_c(q)$.
It is well known that in this regime there exist at least
$q$ distinct shift-invariant extremal Gibbs measures,
namely the monochromatic ordered phases
$\nu_{\beta,q}^1,\dots,\nu_{\beta,q}^q$,
and that they are distinct.
Fix $i\neq j$.
Since $\nu_{\beta,q}^i$ and $\nu_{\beta,q}^j$
are Gibbs measures for the same shift-invariant
finite-range potential, the relative entropy density
between them vanishes:
\[
\ent_*(\nu_{\beta,q}^i \mid \nu_{\beta,q}^j)=0.
\]

By Theorem~\ref{theo-GCB-PREP},
a shift-invariant measure satisfying Gaussian concentration
cannot admit another distinct shift-invariant measure
with zero relative entropy density.
Since $\nu_{\beta,q}^i \neq \nu_{\beta,q}^j$,
it follows that none of the measures
$\nu_{\beta,q}^1,\dots,\nu_{\beta,q}^q$
can satisfy Gaussian concentration.
\end{proof}

\begin{remark}
For $\beta>\beta_c(q)$ the only extremal shift-invariant
Gibbs measures are the $q$ ordered phases
$\nu_{\beta,q}^1,\dots,\nu_{\beta,q}^q$.
The free boundary condition measure $\nu_{\beta,q}^0$
is a convex combination of these phases and hence not extremal.
At criticality, the structure depends on the order of the phase
transition; we do not address that case here.
\end{remark}

\begin{remark}
For the Potts model, the results of \cite{spinka-harel} substantially strengthen those of \cite{spinka-AoP}. The methods are quite different: in particular, \cite{spinka-harel} does not proceed through spatial mixing, which is the mechanism used in the next subsection. 

Nevertheless, the spatial mixing approach has the advantage of being more flexible and applies to a broader class of models, including systems for which no direct finitary coding construction is currently available.
\end{remark}

\subsubsection{Weak and strong spatial mixing}

Let $Y=(Y_i)_{i\in\Z^d}$ be a finite-valued Markov random field with law $\nu$, supported on the feasible set $\mathsf{Y}\subset B^{\Z^d}$. Recall that for a finite set $\Lambda\Subset\Z^d$, the external nearest-neighbor boundary is
\[
\partial\Lambda:=\{i\in\Lambda^\comp:\mathrm{dist}(i,\Lambda)=1\}.
\]

Write $\mathsf{Y}_{\partial\Lambda}$ for the feasible boundary configurations on $\partial\Lambda$. For finite $\Lambda'\subset\Lambda$ and $z\in\mathsf{Y}_{\partial\Lambda}$, let
\[
\nu_{\Lambda,\Lambda'}^{\,z}(\cdot)
:=
\Law_\nu\bigl(Y_{\Lambda'}\in\cdot\mid Y_{\partial\Lambda}=z\bigr)
\]
for $\nu$-a.e.\ feasible $z$.

We recall two classical notions of spatial mixing.

We say that $\nu$ satisfies \emph{weak spatial mixing with rate $\varrho:\N\to[0,\infty)$} if $\varrho$ is nonincreasing, $\varrho(n)\to0$, and for every finite $\Lambda\Subset\Z^d$, every $\Lambda'\subset\Lambda$, and all feasible $z,z'\in\mathsf{Y}_{\partial\Lambda}$,
\[
\bigl\|\nu_{\Lambda,\Lambda'}^{\,z}-\nu_{\Lambda,\Lambda'}^{\,z'}\bigr\|_{\scriptscriptstyle{\mathrm{TV}}}
\le
|\Lambda'|\,\varrho\bigl(\mathrm{dist}(\Lambda',\partial\Lambda)\bigr).
\]
If in addition $\varrho(n)\le C\e^{-cn}$ for some $c,C>0$ and all $n\ge1$, we say that $\nu$ satisfies \emph{exponential weak spatial mixing}.

We say that $\nu$ satisfies \emph{strong spatial mixing with rate $\varrho:\N\to[0,\infty)$} if $\varrho$ is nonincreasing, $\varrho(n)\to0$, and for every finite $\Lambda\Subset\Z^d$, every $\Lambda'\subset\Lambda$, and all feasible $z,z'\in\mathsf{Y}_{\partial\Lambda}$,
\[
\bigl\|\nu_{\Lambda,\Lambda'}^{\,z}-\nu_{\Lambda,\Lambda'}^{\,z'}\bigr\|_{\scriptscriptstyle{\mathrm{TV}}}
\le
|\Lambda'|\,\varrho\Bigl(\mathrm{dist}\bigl(\Lambda',\{i\in\partial\Lambda:z_i\neq z'_i\}\bigr)\Bigr).
\]
If $\varrho(n)\le C\e^{-cn}$ for some $c,C>0$ and all $n\ge1$, we say that $\nu$ satisfies \emph{exponential strong spatial mixing}.

As a direct consequence of Theorem~1.1 in \cite{spinka-AoP} and Theorem~\ref{main-thm-ffiid}, we obtain the following.

\begin{theo}\label{thm:SSM-GCB}
Let $d\ge1$ and let $Y=(Y_i)_{i\in\Zd}$ be a random field taking values in a finite set $B$.
If $Y$ satisfies exponential strong spatial mixing, then $Y$ satisfies Gaussian concentration.
If $d=2$ and $Y$ satisfies exponential weak spatial mixing for squares and has no hard constraints, that is, if the topological support of its law is $B^{\Zd}$, then $Y$ also satisfies Gaussian concentration.
\end{theo}

We illustrate this result with one example borrowed from \cite{spinka-AoP}. Further examples, including the hard-core, Widom-Rowlinson, and beach models, are discussed there. In regimes where the relevant spatial mixing property is known, our theorem yields Gaussian concentration. For instance, for the beach model, neither Dobrushin's uniqueness condition nor disagreement percolation applies directly, but \cite{spinka-AoP} establishes sufficient spatial mixing in certain parameter ranges, from which Gaussian concentration follows.

\paragraph{Proper colorings.}
Let $q\ge3$ be an integer. A proper $q$-coloring is a configuration $x\in\{1,\dots,q\}^{\Z^d}$ satisfying $x_i\neq x_j$ whenever $i$ and $j$ are adjacent. The set of proper $q$-colorings defines a subshift of finite type in $\{1,\dots,q\}^{\Z^d}$, and proper colorings arise as ground states of the nearest-neighbor antiferromagnetic Potts model. For $\Lambda\Subset\Z^d$ and a boundary condition $z$ on $\Lambda^\comp$, the finite-volume Gibbs measure is the uniform law on proper $q$-colorings of $\Lambda$ matching $z$ on $\partial\Lambda$.

It is classical, for instance by Dobrushin's uniqueness condition, that the model admits a unique Gibbs measure when $q>4d$, and that this measure satisfies exponential strong spatial mixing. This threshold can be improved to
\[
q>2\alpha d-\gamma,
\]
where
\begin{equation}\label{eq:alpha-gamma-apps}
\alpha^\alpha=\e
\qquad\text{and}\qquad
\gamma:=\frac{4\alpha^3-6\alpha^2-3\alpha+4}{2(\alpha^2-1)}.
\end{equation}
Numerically, $\alpha\approx1.763$ and $\gamma\approx0.47$.

\begin{theo}
For $d\ge2$ and $q>2\alpha d-\gamma$, with $\alpha$ and $\gamma$ as in \eqref{eq:alpha-gamma-apps}, the unique Gibbs measure for proper $q$-colorings of $\Z^d$ satisfies Gaussian concentration.
\end{theo}


\subsection{The thermodynamic jamming limit of the parking process}

We next consider a non-equilibrium example. The simple parking process is a particular instance of the broader class of random sequential adsorption models; see \cite{MR1887532,coletti2024fluctuations}. These models are defined by an irreversible deposition mechanism and therefore fall outside the class of equilibrium models such as Gibbs distributions \cite{RevModPhys.65.1281,MR1824203}.

Let $\Lambda_n=[-n,\dots,n]^d\cap\Z^d$, viewed as an initially empty box. Cars are parked sequentially according to the following rule. At each step, a site $i\in\Lambda_n$ is sampled uniformly among those not previously selected. If all $2d$ nearest neighbors of $i$ are empty, then $i$ becomes occupied; otherwise it remains vacant. Once all sites have been examined, the procedure stops, and the resulting configuration in $\{0,1\}^{\Lambda_n}$ is called the jamming limit of $\Lambda_n$.

Penrose \cite{MR1887532} proved a weak law of large numbers and a central limit theorem for the proportion of occupied sites as $n\to\infty$. Subsequently, Ritchie \cite{MR2205908} introduced the thermodynamic jamming limit, that is, an infinite-volume random field $Y=(Y_i)_{i\in\Z^d}$, and showed that it can be constructed as a finitary coding of i.i.d.\ random variables $X_i\sim\mathrm{Unif}[0,1]$, with exponentially decaying coding radius.

As a consequence, Gaussian concentration for the random field $Y$ follows from Theorem~\ref{main-thm-ffiid}. This is substantially stronger than Proposition~2.4 in \cite{coletti2024fluctuations}, which applies only to the proportion of occupied sites.

\subsection{Random fields arising as limiting distributions of probabilistic cellular automata}

We now return to the mechanism underlying many of the preceding examples, namely finitary codings produced by coupling-from-the-past constructions for probabilistic cellular automata (PCA). Throughout this subsection we use the notation of \cite{spinkaEJP}.

Let $B$ be a non-empty finite set and let $A$ be a finite set. A PCA is specified by finite sets $F,F'\Subset\Z^d$, a family of i.i.d.\ random variables $(W_{v,t})_{v\in\Z^d,\ t\in\Z}$ taking values in $A$, and a local update function
\[
f:B^F\times A^{F'}\to B.
\]
Given an initial configuration $\xi\in B^{\Z^d}$ and an initial time $t_0\in\Z$, the time evolution $(\omega^{\xi,t_0}_{v,t})_{v\in\Z^d,\ t\ge t_0}$ is defined recursively by
\begin{align}
\omega^{\xi,t_0}_{v,t_0} &:= \xi_v, \qquad v\in\Z^d,\label{eq:defPCA-app1}\\
\omega^{\xi,t_0}_{v,t+1}
&:=
f\Bigl((\omega^{\xi,t_0}_{v+u,t})_{u\in F},\ (W_{v+u,t})_{u\in F'}\Bigr),
\qquad v\in\Z^d,\ t\ge t_0.\label{eq:defPCA-app2}
\end{align}

The PCA is called uniformly ergodic if, for every $v\in\Z^d$, the coalescence time
\begin{equation}\label{eq:coalescencePCA-app}
\tau_v
:=
\min\bigl\{t\ge0:\omega^{\xi,-t}_{v,0}\ \text{does not depend on }\xi\bigr\}
\end{equation}
is almost surely finite. In this case the stationary field $\omega^*=(\omega_v^*)_{v\in\Z^d}$ is defined by
\begin{equation}\label{eq:FCPCA-app}
\omega_v^*:=\omega^{\xi,-\tau_v}_{v,0},
\qquad v\in\Z^d,
\end{equation}
which is almost surely well defined and independent of $\xi$. Its law is the limiting distribution of the PCA; see Proposition~2.3 in \cite{spinkaEJP}.

The construction \eqref{eq:defPCA-app1}-\eqref{eq:FCPCA-app} is a finitary coding from the i.i.d.\ field $((W_{v,t})_{t<0})_{v\in\Z^d}$ to the stationary law of the PCA. Moreover, the cone structure of the dependence implies the short-range factorization property needed in Theorem~\ref{theo:cone}.

\begin{figure}
\begin{tikzpicture}[scale=1.2]

\draw[->] (0,-5) -- (0,2) node[right] {$t$};
\draw[->] (-6,0) -- (6,0) node[below] {$\mathbb{Z}^d$};

\node[fill=black,circle,inner sep=2pt,label=above:$i$] (i) at (-4,0) {};
\node[fill=black,circle,inner sep=2pt,label=above:$k$] (k) at (-1,0) {};
\node[fill=black,circle,inner sep=2pt,label=above:$\ell$] (ell) at (2,0) {};

\draw[thick,blue] (-1,0) -- (-4,-3);
\draw[thick,blue] (-1,0) -- (2,-3);

\draw[thick,red] (2,0) -- (0.5,-1.5);
\draw[thick,red] (2,0) -- (3.5,-1.5);

\begin{scope}[opacity=0.12]
  \fill[blue]  (-1,0) -- (-4,-3) -- (2,-3) -- cycle;
  \fill[red]   (2,0)  -- (0.5,-1.5) -- (3.5,-1.5) -- cycle;
\end{scope}

\end{tikzpicture}
\caption{Cone of influence for points $k$ and $\ell$ in $\mathbb{Z}^d\times\mathbb{Z}$ (here $d=1$), in the case where $\max\{\|\ell-i\|,\|k-i\|,\|\ell-k\|\}=\|\ell-i\|$.}
\label{fig:cone}
\end{figure}

\begin{theo}\label{thm:PCA-cone}
Let $\mu$ be the limiting distribution of a uniformly ergodic PCA. Then $\mu$ is a finitary coding of an i.i.d.\ random field satisfying the short-range factorization property with $\alpha=1/2$.
\end{theo}

\begin{proof}
Let $s:=F\cup F'$. Define $S_0=s$ and recursively
\[
S_n:=\bigcup_{i\in S_{n-1}}(s+i),
\qquad n\ge1.
\]
For $v\in\Z^d$, the cone of influence of $\omega_v^*=\omega^{\xi,-\tau_v}_{v,0}$ is the random set
\[
C_v:=\bigcup_{t=0}^{\tau_v}\ \bigcup_{i\in S_t}(i+v,t).
\]
By construction, $\omega_v^*$, and hence its coding radius, is measurable with respect to the variables $W_{i,t}$ with $(i,t)\in C_v$. Writing
\[
\overline W_i:=(W_{i,t})_{t<0},
\qquad i\in\Z^d,
\]
we see that $\mu$ is a finitary factor of the i.i.d.\ field $\overline W=(\overline W_i)_{i\in\Z^d}$.

Now let $k,\ell,i\in\Z^d$ satisfy
\[
\max\{\|\ell-i\|,\|k-i\|,\|\ell-k\|\}=\|\ell-i\|.
\]
Then the indicators
\[
\1\{\|k-i\|_\infty\le r_k(\overline W)\}
\quad\text{and}\quad
\1\bigl\{\tfrac12\|k-\ell\|_\infty\le r_\ell(\overline W)\bigr\}
\]
depend on disjoint sets of input variables, as illustrated in Figure~\ref{fig:cone}, and are therefore independent. This is precisely the short-range factorization property with $\alpha=1/2$.
\end{proof}

As a consequence, any uniformly ergodic PCA whose coding volume has finite first moment satisfies Gaussian concentration by Theorem~\ref{theo:cone}.

\subsection{Left finitary processes}\label{sec:LFP}

We now turn to one-dimensional processes. Throughout, stochastic processes are viewed as probability measures on $B^{\Z}$, where $\Z$ plays the role of the time axis. We introduce a general class of processes, which we call \emph{left finitary processes}. Closely related notions appear in the literature under the name of unilateral codings \cite{del1990bernoulli,ornstein1975unilateral}.

Let $A$ and $B$ be standard Borel spaces. Suppose that the $B$-valued process $Y=(Y_i)_{i\in\Z}$ is obtained as a stationary coding of an $A$-valued process $X=(X_i)_{i\in\Z}$, and let $\varphi$ denote the coding map. For $x\in A^{\Z}$, define the \emph{left coding radius at the origin} by
\[
r_\varphi^-(x)
:=
\inf\Bigl\{r\in\N_0:\ \forall y\in A^{\Z},\
y_{-r}^0=x_{-r}^0\Rightarrow \varphi(y)_0=\varphi(x)_0\Bigr\}
\in\N_0\cup\{\infty\}.
\]
We say that $Y$ is a \emph{left finitary coding} of $X$ if $r_\varphi^-$ is almost surely finite.

The next statement is an immediate consequence of Theorem~\ref{theo:cone}.

\begin{theo}\label{theo:LFC}
If $Y$ is a left finitary coding of an i.i.d.\ process $X$, then for every local function $f:B^{\Z}\to\R$ satisfying the bounded-difference property,
\begin{equation}\label{eq:LFC-app}
\log \E\big[\exp\{\lambda(f(Y)-\E f(Y))\}\big]
\le
3\lambda^2
\bigl(\E[2r_\varphi^-(X)+1]\bigr)^2
\|\delta f\|_2^2,
\qquad \forall\,\lambda>0.
\end{equation}
\end{theo}

\begin{proof}
A left finitary coding satisfies the short-range factorization property with $\alpha=1$; indeed, because the coding is one-sided, the relevant dependence events are functions of disjoint blocks of the input process. The result therefore follows from Theorem~\ref{theo:cone}.
\end{proof}

Equivalently, if $Y=\varphi(X)$ is left finitary, then there exist a measurable map $\psi$ and a stopping time $\tau$ such that
\[
Y_0=\psi(X_0,X_{-1},\dots,X_{-\tau}).
\]
Thus left finitary processes can be viewed as random generalizations of moving averages of finite order.

\paragraph{Coupling from the past.}
A natural source of left finitary codings is provided by coupling-from-the-past (CFTP) constructions. Consider a stochastic recursive sequence
\[
Y_i=f_i(Y_{-\infty}^{i-1};X_i),
\qquad i\in\Z,
\]
driven by an i.i.d.\ process $X=(X_i)_{i\in\Z}$ on a standard Borel space $A$, with values in a standard Borel space $B$, and assume the family $(f_i)$ is stationary up to translation. Similarly as we did for PCA, let us define, for any $y\in B^{\Z}$ and $t_0\in\Z$, the process $(Y^{[y],t_0}_t)_{t\in\Z}$ as
\begin{align*}
Y^{[y],t_0}_t&=y_t\,\,,\,\,\,\,\,t\le t_0\\
Y^{[y],t_0}_t&=f_t(y_{-\infty}^{t_0}Y^{[y],t_0}_{t_0+1}\ldots Y^{[y],t_0}_{t_0-1},X_t)\,\,,\,\,\,\,\,t\ge t_0+1.
\end{align*}
Define the regeneration time
\[
\theta
:=
\inf\bigl\{k\ge0:Y_0^{[y],-k}\ \text{does not depend on }y\bigr\}.
\]
If $\theta<\infty$ almost surely, then the process is a left finitary coding of the i.i.d.\ input. Hence Theorem~\ref{theo:LFC} yields the following.

\begin{cor}\label{cor:CFTP}
If $Y$ is obtained by a CFTP algorithm and $\E\theta<\infty$, then $Y$ satisfies the Gaussian concentration bound \eqref{eq:LFC-app}, with $\theta$ in place of $r_\varphi^-$.
\end{cor}

We now illustrate this general principle in several classical one-dimensional settings.

\subsection{Markov chains}\label{sec:Markov}

We now specialize to Markov chains. While left finitary codings and coupling-from-the-past constructions provide a natural source of examples, the Markovian setting admits a more precise and essentially complete characterization, obtained by combining our abstract results with several known equivalences.

Gaussian concentration for Markov chains has been studied in several works \cite{marton1996measure,samson2000,redig2009concentration,paulin2015concentration,chazottes2023gaussian}. More recently, \cite{dedecker/gouezel/2015} proved that a stationary Markov chain satisfies Gaussian concentration if and only if it is geometrically ergodic, that is, there exists $\rho\in(0,1)$ such that for every state $b$ there is a constant $C_b$ satisfying
\[
\|P^n(b,\cdot)-\pi\|_{\scriptscriptstyle{\mathrm{TV}}}
\le
C_b\,\rho^n.
\]
This is strictly weaker than uniform ergodicity; see, for instance, the Toboggan chain discussed below. Subsequently, \cite{havet2020quantitative} obtained Gaussian concentration under geometric ergodicity, with an explicit but typically hard-to-compute concentration constant.

Our contribution is to place these results within a broader structural framework and to relate them to finitary codings and return-time properties, leading to a collection of equivalent characterizations.

\subsubsection{Geometrically ergodic Markov chains}

We say that a chain has \emph{exponential return times} if for every $b\in B$ there exist $c,C>0$ such that
\[
\mathbb P(\tau_b>k\mid Y_0=b)\le C\e^{-ck},
\qquad \forall k\in\N,
\]
where $\tau_b:=\inf\{k\ge1:Y_k=b\}$.

\begin{theo}\label{prop:equiv_Markov}
Let $Y=(Y_n)_{n\in\mathbb Z}$ be a stationary, irreducible, and aperiodic Markov chain with countable state space $B$, transition matrix $P$, and unique stationary distribution $\pi$. Then the following are equivalent:
\begin{enumerate}[label=\textup{(\arabic*)}]
\item $Y$ is geometrically ergodic;
\item $Y$ satisfies Gaussian concentration;
\item $Y$ has exponential return times;
\item $Y$ is a finitary coding of an i.i.d.\ process with exponentially decaying coding radius;
\item $Y$ is a coding of an i.i.d.\ process;
\item $Y$ is finitarily isomorphic to an i.i.d.\ process.
\end{enumerate}
\end{theo}

\begin{remark}
Many further equivalent formulations are known. For instance, \cite{gallegos2024equivalences} lists 27 equivalent characterizations of geometric ergodicity. Moreover, \cite{bradley/2005} shows that geometric ergodicity is equivalent to exponential $\beta$-mixing; see also \cite{shields/1996}.
\end{remark}

\begin{remark}
Foss and Tweedie \cite{foss/tweedie/1998} proved that for Markov chains on general state spaces, the existence of a CFTP algorithm is equivalent to uniform ergodicity. Thus Theorem~\ref{prop:equiv_Markov} shows that geometrically ergodic chains that are not uniformly ergodic provide natural examples of finitary processes that cannot arise from a CFTP construction.
\end{remark}

\begin{proof}
The equivalence (1) $\Leftrightarrow$ (2) is due to \cite{dedecker/gouezel/2015}.

(2) $\Rightarrow$ (3): fixing $b\in B$ and applying the Gaussian tail bound \eqref{Gaussian-tails} to the empirical mean of $\1_{\{Y_i=b\}}$ with deviation level $\pi(b)/2$ gives
\[
\mathbb P(\tau_b>n)
=
\mathbb P\Bigl(\sum_{i=1}^n\1_{\{Y_i=b\}}=0\Bigr)
\le
\exp(-c_b n)
\]
for some $c_b>0$. Stationarity then yields exponential tails for the return time from $b$.

(3) $\Rightarrow$ (4): this is Theorem~1 of \cite{angel2021markov}.

(4) $\Rightarrow$ (5) is immediate.

(5) $\Rightarrow$ (3): this implication is essentially due to Smorodinsky and appears in \cite{rudolph/1982}; it also follows from \cite{angel2021markov}. Indeed, for fixed $b\in B$, the indicator process
\[
Z_i:=\1_{\{Y_i=b\}},\qquad i\in\Z,
\]
is finitary, and Proposition~3 of \cite{angel2021markov} then yields exponential tails for its inter-arrival times, hence for return times to $b$ in $Y$.

(4) $\Rightarrow$ (2) follows from Theorem~\ref{main-thm-ffiid}.

Finally, (6) $\Leftrightarrow$ (3) is due to \cite{rudolph/1982} under a finite-entropy assumption, and was recently reproved by \cite{spinka2025new} without any entropy restriction.
\end{proof}

\subsubsection{Further remarks on Markov chains}

\paragraph{Explicit bounds under uniform ergodicity.}
For Markov chains on general state spaces, \cite{foss/tweedie/1998} showed that CFTP is equivalent to uniform ergodicity. Thus uniformly ergodic chains satisfy Corollary~\ref{cor:CFTP}. Under a Doeblin condition, there exist $m\ge1$, a probability measure $\nu$, and $\beta\in(0,1)$ such that
\[
\inf_{x\in B}P^m(x,E)\ge \beta\,\nu(E),
\qquad \forall E\subset B.
\]
In this setting, \cite{foss/tweedie/1998} use the multigamma coupling of \cite{murdoch1998exact}, for which
\[
\frac{\theta}{m}\eqlaw\mathrm{Geo}(\beta).
\]
Hence $\E[\theta]=m/\beta$, and Corollary~\ref{cor:CFTP} yields
\[
\log \E\big[\exp\{\lambda(f(Y)-\E f(Y))\}\big]
\le
3\lambda^2(2m/\beta+1)^2\|\delta f\|_2^2,
\qquad \forall\lambda>0.
\]

\paragraph{Explicit bounds under geometric ergodicity.}
Theorem~\ref{prop:equiv_Markov}, combined with Theorem~\ref{main-thm-ffiid}, gives
\begin{equation}\label{eq:GCBinexplicit-app}
\log \E\big[\exp\{\lambda(f(Y)-\E f(Y))\}\big]
\le
2^d\lambda^2\,\E\big[(2r_\varphi(X)+1)^2\big]\,
\|\delta f\|_2^2,
\qquad \forall\lambda>0,
\end{equation}
where $r_\varphi$ is the coding radius associated with a finitary coding of the chain. The construction of \cite{angel2021markov} gives, in principle, explicit tail bounds on $r_\varphi$, though the resulting constants are much less transparent than in the uniformly ergodic case.

\paragraph{A toy example: the Toboggan chain.}
Consider the Markov chain on $B=\N$ with transition matrix
\[
P(0,i)=p_i,\quad i\ge0,
\qquad
P(i,i-1)=1,\quad i\ge1,
\]
where $(p_i)_{i\ge0}$ is a probability distribution on $\N$ with $p_i>0$ for all $i$. This chain is irreducible and aperiodic. It is positive recurrent if and only if
\[
\mu:=\sum_{i\ge0}i\,p_i<\infty,
\]
in which case the stationary distribution is
\[
\pi(j)=\frac{1}{\mu}\sum_{k\ge j}p_k,
\qquad j\ge0.
\]
It satisfies the equivalent properties of Theorem~\ref{prop:equiv_Markov} if and only if there exists $r>1$ such that
\[
\E_0[r^{\tau_0}]<\infty;
\]
see \cite[Theorem~15.1.4]{meyn/tweedie/2012}. However, unless $(p_i)$ has finite support, the chain is not uniformly ergodic, because
\[
P^n(k,0)=0 \quad\text{for all }k>n.
\]

To illustrate \eqref{eq:GCBinexplicit-app}, consider the geometric case $p_i=2^{-i-1}$. Then the associated renewal process
\[
Z_i:=\1_{\{Y_i=0\}},\qquad i\in\Z,
\]
admits a finitary coding by \cite{angel2021markov}. In this simple case, one checks directly from their proof that the coding radius $\theta$ satisfies
\[
\mathbb P(\theta=i)=2^{-(i+1)},
\qquad i\ge0.
\]
Hence $\E[\theta]=1$ and $\E[\theta^2]=3$, yielding an explicit Gaussian concentration constant despite the lack of uniform ergodicity.

\subsubsection{Renewal processes}

A discrete-time renewal process is a binary-valued process in which the distances between successive $1$'s are i.i.d.\ random variables. Let $(f_k)_{k\ge1}$ denote their common distribution. Renewal processes are Markovian only in the geometric case, but they retain many features of the Markov setting. In particular, the indicator process of successive visits to a fixed state in a Markov chain is a renewal process.

Using this connection, one obtains the following.

\begin{proposition}
Let $Y=(Y_n)_{n\in\Z}$ be a renewal process with $\gcd\{k\ge1:f_k>0\}=1$. Then the following are equivalent:
\begin{enumerate}[label=\textup{(\arabic*)}]
\item $Y$ satisfies Gaussian concentration;
\item $\sum_{k\ge1}s^k f_k<\infty$ for some $s>1$;
\item $Y$ is a finitary process with exponentially decaying coding radius;
\item $Y$ is a finitary coding of an i.i.d.\ process.
\end{enumerate}
\end{proposition}

\begin{proof}
The equivalence (1) $\Leftrightarrow$ (2) is proved in \cite[Theorem~3.4]{chazottes2023gaussian}. Implication (2) $\Rightarrow$ (3) is Theorem~2 of \cite{angel2021markov}. Implication (3) $\Rightarrow$ (4) is immediate. Finally, (4) $\Rightarrow$ (2) is Proposition~3 of \cite{angel2021markov}.
\end{proof}

\subsection{Stochastic chains with unbounded memory}

A stochastic chain with unbounded memory is a discrete-time process whose
conditional distribution at time $n$, given the past, may depend on an
unbounded portion of the past rather than on a fixed finite window.
This class contains Markov chains and renewal processes as special cases,
but also includes genuinely non-Markovian processes.
Such processes are also known as \emph{chains with complete connections}
or \emph{$g$-measures}; see \cite{fernandez/maillard/2005}.

\medskip

Let $B$ be a measurable space with sigma-field $\mathcal B$.
A measurable map
\[
g:\mathcal B\times B^{(-\infty,-1]}\to[0,1]
\]
is called a \emph{transition kernel} if

\begin{itemize}
\item for every $x\in B^{(-\infty,-1]}$, the map
$S\mapsto g(S\mid x)$ is a probability measure on $(B,\mathcal B)$,
\item for every $S\in\mathcal B$, the map
$x\mapsto g(S\mid x)$ is measurable.
\end{itemize}

A stationary process $Y=(Y_n)_{n\in\Z}$ with law $\mu$ on $B^\Z$
is said to be compatible with $g$ if for every $n\in\Z$ and every
$S\in\mathcal B$,
\[
\mathbb E_\mu\!\left[\mathbf 1_S(Y_n)\,\middle|\,
Y_{-\infty}^{\,n-1}\right]
=
g\!\left(S\mid Y_{-\infty}^{\,n-1}\right)
\quad\text{$\mu$-a.s.}
\]
When such a stationary compatible process exists, we call it {\em stochastic chain with unbounded memory}.

\medskip

Gaussian concentration for chains with unbounded memory was established
in \cite{chazottes2023gaussian} under suitable regularity assumptions
on the kernel.
In the present paper, we obtain concentration instead through
our general finitary-coding results.
More precisely, whenever the process can be generated by a
coupling-from-the-past (CFTP) algorithm with finite expected regeneration
time $\mathbb E[\theta]<\infty$, Corollary~\ref{cor:CFTP} implies that
the process satisfies Gaussian concentration,
with a constant controlled by $(\mathbb E[\theta])^2$.

The first CFTP construction for chains with unbounded memory
was introduced in \cite{comets_fernandez_ferrari_2002}.
Assume that $B$ is countable.
Define
\begin{align*}
\alpha_0
&:=
\sum_{b\in B}\,\inf_{x\in B^{(-\infty,-1]}} g(b\mid x),\\
\alpha_k
&:=
\inf_{a_{-k}^{-1}\in B^{[-k,-1]}}
\sum_{b\in B}\,
\inf_{x\in B^{(-\infty,-k-1]}}
g\big(b\mid x a_{-k}^{-1}\big),
\qquad k\ge1.
\end{align*}
They proved that if
\[
\prod_{k\ge0}\alpha_k>0,
\]
then the corresponding CFTP algorithm has finite expected regeneration time.
We now record a simple observation concerning its exact value.

Let $\theta$ denote the regeneration time of the CFTP construction.
By \cite[Theorem 4.1(iv)]{comets_fernandez_ferrari_2002},
\[
\mathbb P(\theta>m)=\mathbb P(\zeta_m=0), \quad m\ge0,
\]
where $(\zeta_m)_{m\ge0}$ is a Markov chain on $\mathbb N$
started at $0$ and with transition probabilities
\[
\mathbb P(\zeta_{m+1}=i+1\mid \zeta_m=i)=\alpha_i,
\qquad
\mathbb P(\zeta_{m+1}=0\mid \zeta_m=i)=1-\alpha_i.
\]
Thus, $\zeta$ either jumps to $0$ or increases by one.
The event $\{\tau_0=\infty\}$ that the chain never returns to $0$
corresponds to the event that it keeps increasing forever, which occurs
with probability
\[
\mathbb P(\tau_0=\infty)=\prod_{k\ge0}\alpha_k.
\]
Using the identity 
$\sum_{m\ge0}\mathbb P(\zeta_m=0)
=\mathbb P(\tau_0=\infty)^{-1}$ (see \cite[(A.5)]{bressaud_fernandez_galves_1999a}),
we obtain
\[
\mathbb E[\theta]
=
\sum_{m\ge0}\mathbb P(\theta>m)
=
\sum_{m\ge0}\mathbb P(\zeta_m=0)
=
\frac{1}{\prod_{k\ge0}\alpha_k}.
\]

The following result is therefore a consequence of Corollary~\ref{cor:CFTP}.

\begin{proposition}
Let $Y$ be a stationary process with countable alphabet $B$
and transition kernel $g$.
If
\[
\prod_{k\ge0}\alpha_k>0,
\]
then $Y$ satisfies Gaussian concentration.
More precisely, if $\theta$ denotes the regeneration time of the CFTP
construction, then the Gaussian concentration constant $C$
in \eqref{def-GCB} is proportional to
\[
\big(\mathbb E[\theta]\big)^2
=
\Bigg(\prod_{k\ge0}\alpha_k\Bigg)^{-2}.
\]
\end{proposition}

\medskip

The existence of CFTP constructions with finite expected regeneration
time has since been extended far beyond the setting of
\cite{comets_fernandez_ferrari_2002};
see, for instance,
\cite{gallo/2011,desantis/piccioni/2010,gallo/garcia/2013,
gallo2014attractive,de2025new}.
In all these situations, whenever the expected coalescence time is finite,
Gaussian concentration follows from Corollary~\ref{cor:CFTP}.

Finally, although Corollary~\ref{cor:CFTP} applies to general alphabets,
existing CFTP constructions for chains with unbounded memory appear,
to the best of our knowledge, to be available only for countable alphabets.

\section{Open problems}\label{sec:remarks-and-questions}

\subsection{Does Gaussian concentration imply finitary coding by an i.i.d. field?}

Theorem \ref{main-thm-ffiid} raises a natural question. Does Gaussian concentration imply that a shift-invariant random field has to be a finitary coding of an i.i.d. process, under a suitable moment condition on the coding volume?
The guiding intuition is that Gaussian concentration imposes strong structural constraints on the dependence structure of the field. More precisely, we ask the following question.

\begin{question}
Let $Y=(Y_i)_{i\in\Zd}$ be a shift-invariant random field satisfying Gaussian concentration.
Does Gaussian concentration entail the existence of a finitary i.i.d. coding, under an appropriate moment condition on the coding volume?
\end{question}

Equivalently, is it true that if for every coding $\varphi$ and every i.i.d.\ process $X$ such that $Y\eqlaw \varphi(X)$ one has
\[
\E\big(|B_\infty(0,r\!_\varphi)|\big)=\infty,
\]
then $Y$ cannot satisfy Gaussian concentration?
As shown in Proposition~\ref{prop:Ising-at-betac}, this is indeed the case for the Ising model ($d\ge 2$) at $\beta=\beta_c$. Additionally, as discussed above, the conjecture applies to countable-state Markov chains and renewal processes.

When $Y$ takes values in a finite alphabet, one may further ask whether the i.i.d. random field used for the coding can also be chosen to be finite valued.

\subsection{Polynomial coding tails and sharpness of moment conditions}\label{sec:heavyTail}

In all examples in dimension $d \geq 2$ discussed above, the assumptions of Theorem~\ref{main-thm-ffiid} (finite second moment of the coding volume) and Theorem~\ref{theo:cone} (finite first moment) are satisfied with substantial room to spare. Indeed, the coding radius typically exhibits exponential or stretched-exponential tails.

This raises the question of whether these moment conditions are close to optimal. In particular, it is natural to ask whether one can construct examples that lie near the boundary of these assumptions.

\begin{question}
Do there exist Gibbs measures in dimension $d \geq 2$ that are finitary codings of i.i.d.\ random fields, for which the coding radius has polynomially decaying tails, while the coding volume still has a finite first or second moment?
\end{question}

Such examples would provide a natural testing ground for the sharpness of our results. Heuristically, if the coding radius has tail of order $r^{-\alpha}$, then the integrability of the coding volume depends on the relation between $\alpha$ and the dimension $d$, suggesting the existence of borderline regimes.

A natural direction is to investigate models with slow decay of correlations, for instance when correlations are bounded below by a polynomial rate, 
since exponential tails of the coding radius imply exponential decay of correlations, as distant regions are independent unless the coding radii bridge the separation, an event whose probability decays exponentially in the distance.
The long-range Ising model provides a particularly promising class of examples in this direction.

In this direction, it is shown in a recent PhD thesis that if the coupling constants of the long-range Ising model decay like $|i-j|_1^{-\alpha}$ with $\alpha>d$, then the model is a finitary coding of an i.i.d.\ random field whenever $\alpha>2d$ and the inverse temperature $\beta$ is sufficiently small. The proof is outlined in an appendix of that work \cite{Faipeur2025}.

A related question, raised by Spinka \cite{spinka-AoP}, concerns the existence of finitary codings with good tail behavior under mixing assumptions.

\begin{question}
Does exponential weak spatial mixing imply the existence of a finitary coding of an i.i.d.\ random field with finite expected coding radius?
\end{question}

A positive answer would, in combination with coupling-from-the-past constructions, imply Gaussian concentration via Theorem~\ref{theo:cone}. More generally, this question highlights the broader problem of relating quantitative mixing properties to the tail behavior of coding radii.
\subsection{Coding volume with infinite first moment}

We have seen examples in which any finitary coding necessarily has a coding volume with infinite first moment, and for which not only Gaussian concentration fails, but even a moment concentration bound of order $2$ is impossible. A prominent example is the Ising model in dimension $d \geq 2$ at the critical temperature.

Gaussian concentration is a particularly strong form of concentration, and its complete failure in such examples highlights the need to consider weaker notions. It is therefore natural to ask whether some form of concentration may still persist when the coding volume has heavy tails. For instance, one may ask whether moments can still be controlled up to a certain order, or whether all moments can be bounded with constants growing sufficiently fast to preclude exponential moment bounds.

Examples exhibiting intermediate behavior are known. In particular, for the Ising model in dimension $d \geq 2$ at sufficiently low temperature, one obtains stretched-exponential concentration bounds \cite{redig2009concentration,chazottes/collet/kulske/redig/2007}. This suggests that the strength of concentration should be closely related to the tail behavior of the coding volume.

More precisely, we say that a random field $Y$ satisfies a moment concentration bound of order $2p$, with $p \in \N$, if there exists $C_p > 0$ such that for every local function $f$ with the bounded-differences property,
\[
\E\big[(f(Y) - \E f(Y))^{2p}\big] \leq C_p \, \|\delta f\|_p^{2p}.
\]

This leads to the following question.

\begin{question}
Let $Y = \varphi(X)$, where $X$ is an i.i.d.\ random field and $\varphi$ is a finitary coding. To what extent can the strength of concentration for $Y$ be characterized in terms of the tail behavior of the coding volume $|B_\infty(0,X)|$? In particular, which moment concentration bounds can be expected when $\E(|B_\infty(0,X)|)=\infty$?
\end{question}



\section*{Acknowledgments}

The authors thank Aernout van Enter, Corentin Faipeur, S\'ebastien Gou\"ezel, and Frank Redig for helpful comments that significantly improved the clarity and presentation of the paper.

D.~Y.~T. and S.~G. gratefully acknowledge CNRS and \'Ecole Polytechnique for supporting their visits to CPHT, funding a one-month stay in 2022 and another in 2024. \\
J.-R.~C. gratefully acknowledges financial support from the R\'eseau Math\'ematique Franco-Br\'esilien
(\url{https://www.rfbm.fr/}) and the IRP NP-Strong (Non-perturbative methods in strongly coupled field theories and statistics).\\
S.~G. was supported by FAPESP through grants 2023/13453-5, 2024/06341-9 and CNPq through grants 441884/2023-7 and 314909/2023-0.

\bibliographystyle{abbrv}
\bibliography{Concentration_fields}

\end{document}